\documentclass[11p]{amsart}
\usepackage{mathtools}
\usepackage{amssymb}
\usepackage{amsfonts}
\usepackage{amsthm}
\usepackage{mathrsfs}
\usepackage{polynom}
\usepackage{mathbbol}
\usepackage{latexsym,amscd}
\usepackage[all]{xy}
\usepackage{color}

\newtheorem{thm}{Theorem}
\newtheorem{prop}{Proposition}
\newtheorem{lem}{Lemma}
\numberwithin{equation}{section}
\numberwithin{prop}{section}
\numberwithin{lem}{section}
\numberwithin{thm}{section}
\newtheorem{cor}{Corollary}
\numberwithin{cor}{section}

\theoremstyle{definition}
\newtheorem{defn}{Definition}
\numberwithin{defn}{section}

\newtheorem{example}{Example}

\newtheorem{rem}{Remark}
\numberwithin{rem}{section}

\newcommand{\nop}[1]{{}^{\scriptscriptstyle{\circ}}_{\scriptscriptstyle{\circ}}{#1}{}^{\scriptscriptstyle{\circ}}_{\scriptscriptstyle{\circ}}}

\def \<{\left<}
\def \>{\right>}

\def \a{\alpha }

\def \b{\beta }

\newcommand{\bea}{\begin{eqnarray}}
\newcommand{\eea}{\end{eqnarray}}
\newcommand{\be}{\begin {equation}}
\newcommand{\ee}{\end{equation}}

\newcommand{\h}{\mathfrak{h}}
\newcommand{\wt}{{\rm {wt} }   }

\newcommand{\hh}{\hat {\frak h} }

\newcommand{\Z}{\Bbb Z}

\newcommand{\xa}[2]{x^{\tilde{\nu}}_{\alpha_{#1}}\left(#2\right)}

\newcommand{\hnu}{\hat{\nu}}
\newcommand{\n}{\mathfrak{n}}

\newcommand{\ta}[1]{\tau_{\gamma_{#1}}}
\newcommand{\on}{\overline{\n}[\hnu]}

\newcommand{\tp}[1]{\psi_{\gamma_{#1}}\tau_{\gamma_{#1}}}

\newcommand{\ZZ}{\mathbb{Z}}
\newcommand{\CC}{\mathbb{C}}
\newcommand{\QQ}{\mathbb{Q}}

\newcommand{\ea}[2]{\left(e^{\alpha_{#1}}\right)^{\hnu}_{#2}}
\newcommand{\tnu}[0]{\tilde{\nu}}

\usepackage{tikz-cd}

\begin{document}
\title{Principal Subspaces of Twisted Modules for Certain Lattice Vertex Operator Algebras}
\author{Michael Penn, Christopher Sadowski, and Gautam Webb}

\begin{abstract} 
This is the third in a series of papers studying the vertex-algebraic structure of principal subspaces of twisted modules for lattice vertex operator algebras. We focus primarily on lattices $L$ whose Gram matrix contains only non-negative entries. We develop further ideas originally presented by Calinescu, Lepowsky, and Milas to find presentations (generators and relations) of the principal subspace of a certain natural twisted module for the vertex operator algebra $V_L$. We then use these presentations to construct exact sequences involving this principal subspace, which give a set of recursions satisfied by the multigraded dimension of the principal subspace and allow us to find the multigraded dimension of the principal subspace.   
\end{abstract}

\maketitle

\section{Introduction}
Principal subspaces of standard (highest weight integrable modules) for affine Lie algebras have received considerable attention and study since first being defined and studied by Feigin and Stoyanovsky in \cite{FS1}-\cite{FS2}, and the theory of principal subspaces has been developed by many authors, including \cite{G}, \cite{AKS}, \cite{FFJMM}, \cite{Bu1}-\cite{Bu3},\cite{Ka1}-\cite{Ka2}, \cite{MPe}, and many others. We note that certain ``commutative" analogues of principal subspaces have been also defined and studied in \cite{T1}-\cite{T3}, \cite{J1}-\cite{J2}, \cite{JP}, \cite{Pr}, \cite{P} and others, and quantum analogues have been studied in \cite{Ko1}-\cite{Ko2}.
We base our approach in this paper on the vertex-algebraic methods developed in \cite{CLM1}-\cite{CLM2}, \cite{CalLM1}-\cite{CalLM4}, \cite{S1}-\cite{S2}, \cite{C1}, and \cite{PS1}-\cite{PS2}. We note that this paper uses many of the techniques developed in these works, but that the principal subspace we study is a twisted analogue of the ``commutative" structures mentioned above.

This paper is the third in a series of papers by the authors studying the vertex-algebraic structure of principal subspaces of twisted modules for lattice vertex operator algebras (the first two papers being \cite{PS1} and \cite{PS2}). The study of principal subspaces of twisted modules for lattice vertex operator algebras was initiated by Calinescu, Lepowsky, and Milas in \cite{CalLM4}. In \cite{CalLM4}, the authors introduced the notion of principal subspace of a twisted module for a lattice vertex operator algebra, and studied the principal subspace of the basic module for the twisted affine Lie algebra $A_2^{(2)}$. Using ideas in \cite{CLM1}-\cite{CLM2} and \cite{CalLM1}-\cite{CalLM3}, the authors gave a presentation (generators and relations) of the principal subspace and proved that certain natural relations form a complete set of relations for the principal subspace. Using these presentations, the authors then constructed exact sequences among the principal subspace, which yields a recursion satisfied by the multigraded dimension of the principal subspace. Solving this recursion, the authors were able to find the multigraded dimension of the principal subspace, which is related to the generating function for partitions of a positive integer into odd and distinct parts (cf. Corollary 7.4 in \cite{CalLM4}). The ideas in \cite{CalLM4} were later extended to study principal subspaces of standard modules for the twisted affine Lie algebras of type $A_{2n}^{(2)}$ in \cite{CalMPe}, the twisted affine Lie algebra of type $D_4^{(3)}$ in \cite{PS1}, and the twisted affine Lie algebras of type $A_{2n+1}^{(2)}, D_n^{(2)},$ and $E_6^{(2)}$ in \cite{PS2}. We note that the important disctinction between the $A_{2n}^{(2)}$ case and the $D_4^{(3)}$ and the $A_{2n+1}^{(2)}, D_n^{(2)},$ and $E_6^{(2)}$ cases is that the Dynkin diagram automorphism used to construct the twisted module in the $A_{2n}^{(2)}$ case needed to have its order doubled in order to be extended to the underlying lattice vertex operator algebra, whereas in the other cases this was not necessary. This doubling significantly complicated the defining relations of the principal subspace.  We also note here that, in the case of $A_4^{(2)}$, modularity properties of the character are discussed in \cite{CalMPe}.

In this work, we assume $L=\ZZ\alpha_1 \oplus \dots \oplus \ZZ \alpha_D$ is a positive definite even lattice with non-degenerate $\ZZ$-bilinear form $\left< \cdot , \cdot \right>$ satisfying $\left< \alpha_i , \alpha_j \right> \ge 0$ for all $1 \le i,j \le D$, and we assume that $\nu:L \rightarrow L$ is an automorphism of $L$ which permutes the $\alpha_i$ for $1 \le i \le D$. In this setting, we note that we may view $\nu$ in terms of $d$ disjoint cycles from the symmetric group $S_D$, and we take $\alpha^{(i)}$ to be a representative from each orbit we obtain this way for $1 \le i \le d$, and let $l_i$ to be the number of elements in this orbit (or, alternatively, $l_i$ is the length of the cycle). We extend the ideas developed in \cite{CalLM4} to construct a twisted module for the lattice vertex algebra $V_L$, which we call $V_L^T$, and study its principal subspace, which we call $W_L^T$. We note here that this case is in many ways similar to the $A_{2n}^{(2)}$, where we double the order of the automorphism to proceed with our construction.  We provide a set of defining relations for the principle subspace and prove that this is a complete set of defining relations, but note importantly that we require new relations that were not needed in past work, which follow from the theory of twisted vertex operators in \cite{DLeM} (these relations were certainly present in \cite{CalLM4}, \cite{CalMPe}, and \cite{PS1}-\cite{PS2}, but were trivially true and did not need to be mentioned). The proof of the completeness of these relations is highly nontrivial, and requires the use of what we call a ``generalized Pascal matrix" (see the Appendix for its definition and properties). The remainder of the paper is analogous to \cite{CalLM4}, \cite{CalMPe}, and \cite{PS1}-\cite{PS2}. We use our presentations to construct exact sequences involving our principal subspace, and use these exact sequences to find a set of recursions satisfied by the multigraded dimension of the principal subspace. Solving these recursions yields the multigraded dimension of $W_L^T$:
\begin{equation}
\chi^{'}(\mathbf{x};q) = \sum_{{\bf m} \in (\mathbb{Z}_{\ge 0}^d)}\frac{q^{\frac{{\bf m}^t A {\bf m}}{2}}}{(q^{\frac{k}{l_1}};q^{\frac{k}{l_1}})_{m_1} \cdots (q^{\frac{k}{l_d}};q^{\frac{k}{l_d}})_{m_d} }x_1^{m_1}\cdots x_d^{m_d}
\end{equation}
where $k$ is double the order of $\nu$, where for $\alpha \in L$ we define $$\alpha_{(0)} = \frac{1}{k}(\alpha + \nu \alpha + \nu^2 \alpha + \dots + \nu^{k-1}\alpha),$$  and where $A$ is a $(d\times d)$-matrix defined by 
$$A_{i,j} = k\left< \alpha_{(0)}^{(i)}, \alpha_{(0)}^{(j)}\right>.$$ Finally, we produce some examples at the end of the work which are related to known and conjectured partition identities.

\section{The setting}

In this section we follow the construction of lattice vertex operator algebras and their twisted modules as presented in \cite{L1}, \cite{FLM}, \cite{CalLM4}, \cite{L2}, and \cite{LL}. We begin by working in a general setting.

Suppose that we have a rank $D$ positive-definite even lattice
\be
L=\ZZ\a_1\oplus\ZZ\a_2\oplus\cdots\oplus\ZZ\a_D
\ee
equipped with a symmetric non-degenerate $\ZZ$-bilinear form $\<\cdot,\cdot\>:L\times L\to \ZZ$. We also consider the matrix
\be
A=\left(a_{i,j}\right)=\left(\<\a_i,\a_j\>\right)
\ee
and let 
\be 
L_+ = \mathbb{Z}_{\ge 0}\a_1+\mathbb{Z}_{\ge 0}\a_2+\cdots+\\ \mathbb{Z}_{\ge 0}\a_D
\ee
Suppose that $\nu$ is an isometry of $L$ (with respect to $\<\cdot,\cdot\>$) which preserves $L_+$ (in the sense the $\nu(L^+) \subset L^+$) and has finite order $v$. We begin by showing that this type of isometry can be realized as a permutation of the simple roots $\a_1,\dots, \a_D$.
\begin{lem}\label{PosFix}
If $\beta \in L_+$ and $\nu(\alpha) = \beta$ for some $\alpha \in L$, then $\alpha \in L_+$.
\end{lem}
\begin{proof}
We have that $\alpha = \nu^{v-1}(\nu(\alpha)) = \nu^{v-1}(\beta) \in L_+$ since $\beta \in L_+$ and $\nu$ preserves $L_+$.
\end{proof}
\begin{prop}
For each $1 \le i \le D$, we have that $\nu(\alpha_i) = \alpha_j$ for some $1 \le j \le D$.
\end{prop}
\begin{proof}
Suppose that $\nu(\alpha_i) = \alpha$ for some $\alpha \notin \{\a_1,\dots, \a_D \}$. Since $\alpha \in L_+$ and $\alpha \neq 0$,  there exists $\beta \in L_+$ so that $\a - \b = \a_j$ for some $1 \le j \le D$. Since $\nu$ is an isometry, there exists a nonzero $\gamma \in L$ such that $\nu(\gamma) = \beta$. By Lemma \ref{PosFix}, we have that $\gamma \in L_+$. Hence,
\be
\a_j = \a - \b = \nu(\a_i - \gamma)
\ee
but $\a_i - \gamma \notin L_+$, a contradiction.
(Note: if $\a_i - \gamma \in L_+$, then $\a_i = \gamma$ which implies $\a_j = 0$.)
\end{proof}

\begin{rem} We make note here of the fact that the isometries $\nu$ important in this work are those which preserve $L_+$ and will precisely be those which fix the principal subspaces constructed later in this work. In earlier work \cite{CalLM4}, \cite{CalMPe}, and \cite{PS1}-\cite{PS2}, these automorphisms were given by Dynkin diagram automorphisms. \end{rem}

 We now realize $\nu$ by a permutation, which we will also denote by $\nu$ by decomposing it into $l_d$ disjoint cycles given by
\be\label{perm}
\nu=(1, 2,\cdots, l_1)(l_1+1, l_1+2,\cdots, l_1+l_2)\cdots (l_1+\cdots+l_{d-1}+1,\cdots, l_1+\cdots+l_d),
\ee
where $D=l_1+\cdots+l_d$. By saying that $\nu:L\to L$ is realized by this permutation we mean that
\be
\nu\left(\a_i\right)=\a_{\nu(i)}.
\ee

We now relabel the elements of the $\ZZ$-basis of $L$ to interact more cleanly with the isometry. For each $j$ satisfying $1 \le r \le d$, set
\be
\a_1^{(r)} = \a_{l_1+\cdots+l_{r-1}+1}
\ee
and
\be\label{notation}
\a_j^{(r)}=\a_{l_1+\cdots+l_{r-1}+j}=\nu^{j-1}\a_1^{(r)}.\ee
This allows us to decompose the lattice into orbits of $\nu$ as follows
\be
L=\bigoplus_{r=1}^{d}\bigoplus_{j=1}^{l_r}\ZZ\a_j^{(r)}=\bigoplus_{r=1}^{d}\bigoplus_{j=0}^{l_r-1}\ZZ(\nu^j\a_1^{(r)}).\ee

In addition we will say the $\a_j^{(r)}$ is in the $r^{th}$ cycle of $\nu$.

We continue to follow \cite{CalLM4}, \cite{L1}, and \cite{FLM}. Let $\eta$ be a primitive $k^{\text{th}}$ root of unity, where $k = 2v$ is twice the order of $\nu$. We have two central extensions of $L$ by $\<\eta\>$ which we denote by $\hat{L}$ and $\hat{L}_{\nu}$ with commutator maps $C_0$ and $C$ respectively. In other words, we have two exact sequences (up to equivalence)
\be
1\longrightarrow\left<\eta\right>\longrightarrow\hat{L}\overset{\overline{~~}}{\longrightarrow} L\longrightarrow 1
\ee
and
\be
1\longrightarrow\left<\eta\right>\longrightarrow\hat{L}_{\nu}\overset{\overline{~~}}{\longrightarrow} L\longrightarrow 1
\ee
so that $aba^{-1}b^{-1}=C(\overline{a},\overline{b})$ for $a,b\in\hat{L}$ and $aba^{-1}b^{-1}=C_{0}(\overline{a},\overline{b})$ for $a,b\in\hat{L}_{\nu}$, where, as in \cite{CalLM4}, we define
\be
C(\alpha,\beta) = (-1)^{\left< \alpha,\beta \right>}
\ee
and
\be
C_0(\alpha,\beta) = \prod_{j=0}^{k-1}(-\eta^j)^{\left< \nu^k \alpha, \beta \right>}
\ee
for $\alpha,\beta \in L$.
Following \cite{L1} and \cite{CalLM4}, we let 
\begin{align*}
e:L&\to\hat{L}\\
\a&\mapsto e_{\a}\end{align*}
be a normalized sections of $\hat{L}$ so that 
\be
e_0=1\ee
and
\be\overline{e_{\a}}=\a~\text{ for all }~\a\in L.\ee

 We also have a normalized cocyle $\epsilon_{C_0}$ such that 
\be 
e_{\a}e_{\b}=\epsilon_{C_0}(\a,\b)e_{\a+\b}, \text{ in } \hat{L}
\ee
and 
\be 
\frac{\epsilon_{C_0}(\a,\b)}{\epsilon_{C_0}(\b,\a)}=C_0(\a,\b)
\ee
thus,
\be
e_{\b}e_{\a}=C_0(\a,\b)e_{\a}e_{\b}.
\ee

We have a similar section $e^{\nu}$ and equations concerning $C$ and $\epsilon_C$.  We now lift $\nu$ to an automorphism of $\hat{L}$, denoted by $\hnu$, such that for $a\in\hat{L}$
\be
\overline{\hnu a}=\nu\overline{a}
\ee
and
\be\label{fixed}
\hnu a=a~\text{ if }~ \nu \overline{a}=\overline{a},
\ee
where $\overline{e_{\a}}=\a$ for $\a\in L$. To define this lifting, we make the following particular choices for $C_0$ and $\epsilon_{C_0}$
\be
C_0(\a_i,\a_j)=(-1)^{\<\a_i,\a_j\>}
\ee
and 
\be 
\epsilon(\a_i,\a_j) = \left\{
     \begin{array}{lr}
       1 & \text{ if } i\leq j\\
       (-1)^{\<\a_i,\a_j\>} & \text{ if } i>j.
     \end{array}
   \right.
   \ee

Given $\a\in\{\nu^j\a_1^{(r)}|j\in\mathbb{N}\}=\{\a_j^{(r)}|1\leq j\leq l_r\}$, we define the following lifting of $\nu$ to $\hat{L}$:
\be\label{lifting}
\hnu e_{\a} = \left\{
     \begin{array}{lr}
      e_{\nu\a} & \text{ if $l_r$ is odd} \\
      e_{\nu\a} & \text{ if $l_r$ is even and $\left<\nu^{l_r/2}\a,\a\right>\in 2\ZZ$} \\
       \eta_{2l_r}e_{\nu\a} & \text{ if $l_r$ is even and $\left<\nu^{l_r/2}\a,\a\right>\notin 2\ZZ$},
     \end{array}
   \right.
   \ee
where $\eta_{r}$ is a primitive $r^{th}$ root of unity. Observe that for all $l_j$, we have $\left<\eta_{2l_j}\right>\subset \left<\eta\right>$, so including these roots of unity makes 
sense. Since the conditions defining our lifting will arise later in this work, we give the following definition:

\begin{defn}
We say that $\a_i\in L$ satisfies the \textbf{evenness condition} if $\a_i=\a^{(r)}_j$ for some $0\leq j\leq l_r-1$ and one of the following holds
\begin{enumerate} 
\item $l_r$ is a positive even integer and $\left<\a_i,\nu^{l_r/2}\a_i\right>\in 2\mathbb{Z}$
\item $l_r$ is a positive odd integer.
\end{enumerate}
\end{defn}

Now we take an arbitrary $\a\in L$ such that $\nu\a=\a$ and show that $\hnu e_{\a}=e_{\a}$, so our lifting indeed satisfies (\ref{fixed}). Set 
\be 
\a[r]=\a_1^{(r)}+\cdots+\a_{l_r}^{(r)},\ee
the orbit sum of $\a_1^{(r)}$ under $\left<\nu\right>$
and observe that if $l_r$ is odd we have, for $m\geq 0$, 
\be\begin{aligned}\label{calculation1}
\hnu e_{m\a[r]}&=\hnu(e_{m\a_{1}^{(r)}}\cdots e_{m\a^{(r)}_{l_r}})\\
&=\hnu(e_{m\a_{1}^{(r)}})\cdots \hnu(e_{m\a^{(r)}_{l_r}})\\
&=e_{m\a_{2}^{(r)}}\cdots e_{m\a^{(r)}_{l_r}}e_{m\a_{1}^{(r)}}\\
&=(-1)^{m^2\left(\left<\a_1^{(r)}, \a_2^{(r)}+\cdots + \a_{l_r}^{(r)}\right>\right)}e_{m\a_{1}^{(r)}}\cdots e_{m\a^{(r)}_{l_r}}\\
&=(-1)^{m^2\left(\left<\a_1^{(r)}, \a_2^{(r)}+\cdots + \a_{l_r}^{(r)}\right>\right)}e_{m\a[r]}\\
&=e_{m\a[r]},
\end{aligned}\ee
where the final equality holds because for all $2 \leq j\leq l_r$ we have
\be\label{simplify1}
\left<\a_1^{(r)},\a_j^{(r)}\right>=\left<\nu^{l_r-j+1} \a_1^{(r)},\nu^{l_r-j+1} \a_j^{(r)}\right>=\left<\a_1^{(r)},\a_{l_r-j+2}^{(r)}\right>,
\ee
and thus 
\be\label{simplify2} \left<\a_1^{(r)}, \a_2^{(r)}+\cdots + \a_{l_r}^{(r)}\right>=2\sum_{j=1}^{\frac{l_r+1}{2}}\left<\a_1^{(r)},\a_{j}^{(r)}\right>\in 2\ZZ,\ee
where we have used that fact that $\nu$ is an isometry and $\left<\cdot,\cdot\right>$ is symmetric.

Now we consider the case where $l_r$ is even. Using calculations similar to (\ref{simplify1}) and (\ref{simplify2}) we have 
\be\begin{aligned}\label{calculation2}
\sum_{j=2}^{l_r}\left<\a_1^{(r)},\a_{j}^{(r)}\right>&=\left<\a_1^{(r)},\a_{1+\frac{l_r}{2}}^{(r)}\right>+2\sum_{m=1}^{\frac{l_r}{2}}\left<\a_1^{(r)},\a_{j}^{(r)}\right>\\
&=\left<\a_1^{(r)},\nu^{l_r/2}\a_{1}^{(r)}\right>+2\sum_{m=1}^{\frac{l_r}{2}}\left<\a_1^{(r)},\a_{j}^{(r)}\right>.\end{aligned}\ee
If $\left<\a_1^{(r)},\nu^{l_r/2}\a_{1}^{(r)}\right>\in 2\mathbb{Z}$ we have 
\be \hnu e_{m\a[r]}=e_{m\a[r]}.\ee
with a calculation similar to (\ref{calculation1}). If $\left<\a_1^{(r)},\nu^{l_r/2}\a_{1}^{(r)}\right>\notin 2\mathbb{Z}$ we have 
\be\begin{aligned}\label{calculation3}
\hnu e_{m\a[r]}&=(-1)^{m^2}\eta_{2l_r}^{ml_r}e_{m\a[r]}\\
&=e_{m\a[r]},\end{aligned}\ee
because $\eta_{2l_r}^{l_r}=-1$. Finally, it follows from (\ref{calculation1})-(\ref{calculation3}) that if we take an arbitrary $\a\in L$ such that $\nu\a=\a$ written as
$$\a=m_1\a[1]+\cdots+m_d\a[d],$$
we have 
$$\hnu e_{\a}=e_{\a}$$
as desired.

Now we form the lattice vertex operator algebra, $V_L$, which we recall is characterized by the linear isomorphism
\be
V_L\cong S\left(\hh^{-}\right)\otimes \CC[L],
\ee
where $\h=L\otimes_{\ZZ}\CC$,  $\hh^{-}=\h\otimes t^{-1}\CC[t^{-1}]$, and $\CC[L]$ is the group algebra. We have made use of the linear isomorphism $\CC[L]\cong \CC[\hat{L}]\otimes_{\CC[\<\eta\>]}\CC$.

Recall from \cite{LL} and \cite{FLM} that $V_L$ has a natural weight grading given by 
\be
\wt(h_1(n_1)\cdots h_r(n_r)\iota(e_{\a}))=-(n_1+\cdots+n_r)+\frac{1}{2}\<\a,\a\>.
\ee

The vertex operators on $V_L$ are given by 
\be
Y(h,x)=h(x)=\sum_{n\in\ZZ}h(n)x^{-n-1}
\ee
for $h\in\h$, where we associate $h(n)=h\otimes t^n\in\hh$, and
\be
Y(\iota(e_{\a}),x)=E^{-}(-\a,x)E^{+}(-\a,x)e_{\a}x^{\a},
\ee
where 
\be\label{exp}
E^{\pm}(-\a,x)=\text{exp}\left(\sum_{n\in\pm\ZZ_+}\frac{-\a(n)}{n}x^{-n}\right)\in (\text{End }V_L)[[x,x^{-1}]]
\ee
for $\a\in\h$ and where $\iota$ is the canonical linear isomorphism from $\mathbb{C}[L]$ to the induced $\hat{L}$-module $\mathbb{C}\{L\}$. We will often use $e^{\alpha}$ and $\iota(e_{\alpha})$ interchangably. In general for $v=h_1(n_1)\cdots h_r(n_r)\iota(e_{\a})$ we have
\be
Y(v,x)=\nop{\left(\frac{1}{(n_1-1)!}\left(\frac{d}{dx}\right)^{n_1-1}h_1(x)\right)\cdots\left(\frac{1}{(n_r-1)!}\left(\frac{d}{dx}\right)^{n_r-1}h_r(x)\right)Y(\iota(e_{\a}),x)},
\ee
where by $\nop{\cdot}$ we mean normal ordering as described in \cite{LL}, \cite{FLM}, and others. $V_L$ has the structure of a vertex operator algebra, with vacuum vector
\be
\mathbb{1} = 1 \otimes 1
\ee
and conformal vector
\be
\omega = \frac{1}{2} \sum_{i=1}^D u^{(i)}(-1)^2 \mathbb{1},
\ee
where $\{u^{(1)}, \dots, u^{(D)} \}$ forms an orthonormal basis of $\mathfrak{h}$.
We lift the automorphism $\hnu$ of $\hat{L}$ to an automorphism of $V_L$, which we will also denote by $\hnu$. This is given by $\nu\otimes\hnu$ acting on $S\left(\hh^{-}\right)\otimes\CC[L]$ (cf. \cite{L1} and \cite{CalLM4}).

Now we construct the $\hnu$-twisted module $V_L^T$ of $V_L$ in our setting. We begin by decomposing
\be
\h_{(n)}=\{h\in\h|\nu h=\eta^{n}h\}\subset\h
\ee
for $n\in\ZZ$, and thus
\be
\h=\coprod_{n\in\ZZ/k\ZZ}\h_{(n)},
\ee
where we have identified $\h_{(n)}=\h_{(n\text{ mod }k)}$. Define the projection
\be
P_n:\h\to\h_{(n)}
\ee
and set $h_{(n)}=P_{(n\text{ mod }k)}$ for $h\in\h$ and $n\in\ZZ$. Observe that in general we may write
\be
\h_{(n)}=\{h+\eta^{-n}\nu h+\eta^{-2n}\nu^2 h+\cdots+\eta^{-(k-1)n}\nu^{k-1}h|h\in\h\}
\ee
and
\be 
h_{(n)}=\frac{1}{k}\left(h+\eta^{-n}\nu h+\eta^{-2n}\nu^2 h+\cdots+\eta^{-(k-1)n}\nu^{k-1}h\right),
\ee
for $h\in\h$. Now we move this into our particular setting. Notice that any $\a_1,\dots,\a_D$ can be written in the form
$\a^{(r)}_j$ for some $1\leq r\leq d$ and $1\leq j\leq l_r$. 
Of particular interest will be the zeroth component which we may write
\be
\h_{(0)}=\coprod_{r=1}^{d}\CC\left(\a_1^{(r)}+\cdots+\a_{l_r}^{(r)}\right)
\ee
and similarly the zeroth projection
\be
(\a_j^{(r)})_{(0)}=\frac{1}{l_r}\left(\a_1^{(r)}+\cdots+\a_{l_r}^{(r)}\right),
\ee
for $1\leq r\leq d$ and $1\leq j\leq l_r$.

Consider the $\nu$-twisted affine Lie algebra associated with $\h$ and $\<\cdot,\cdot\>$ (which $\h$ inherits from the bilinear form on the lattice $L$).
\be
\hh[\nu]=\coprod_{n\in\frac{1}{k}\ZZ}\h_{(kn)}\otimes t^n\oplus\CC\mathbf{k}
\ee
with 
\be
[\a\otimes t^m,\b\otimes t^n]=\<\a,\b\>m\delta_{m+n,0}\mathbf{k}
\ee
where $\a\in\h_{(km)}$, $\b\in\h_{(kn)}$, $m,n\in\frac{1}{k}\ZZ$, and $\mathbf{k}\in Z\left(\hh[\nu]\right)$, the center. Observe that this algebra is $\frac{1}{k}\ZZ$-graded by weights with
\be
\wt\left(\a\otimes t^m\right)=-m, \text{ and } \wt(\mathbf{k})=0,
\ee
where $m\in\frac{1}{k}\ZZ$ and $\a\in\h_{(mk)}$. Following \cite{CalLM4} we consider the subalgebras

\begin{equation}
\hh[\nu]^{+}=\coprod_{n>0}\h_{(kn)}\otimes t^n,~~\hh[\nu]^{-}=\coprod_{n<0}\h_{(kn)}\otimes t^n
\end{equation}
as well as
\begin{equation}
\hh[\nu]_{\frac{1}{k}\mathbb{Z}}=\hh[\nu]^{+}\oplus\hh[\nu]^{-}\oplus\mathbb{C}\mathbf{k}.
\end{equation}
We now form the induced module
\begin{equation}
S[\nu]=U\left(\hh[\nu]\right)\otimes_{U\left(\coprod_{n\geq0}\h_{(kn)}\otimes t^n\oplus\mathbb{C}\mathbf{k}\right)}\mathbb{C},
\end{equation}
where $\coprod_{n\geq0}\h_{(kn)}$ acts trivially on $\mathbb{C}$ and $\mathbf{k}$ acts as $1$. We will make use of the fact that this is linearly isomorphic to $S\left(\hh[\nu]^{-}\right)$ and give it the natural $\QQ$-grading such that
\begin{equation}\begin{aligned}
\wt~1=&\frac{1}{4k^2}\sum_{j=1}^{k-1}j(k-j)\text{dim }\h_{(j)}
\end{aligned}\end{equation}

We now move towards constructing the $\hnu$-twisted $V_L$ module $V_L^T$. As in \cite{L1} and \cite{CalLM4}, set
\be 
N=\left(1-P_0\right)\h\cap L=\{\a\in L|\<\a,\h_{(0)}\>=0\},
\ee
\be
M=(1-\nu)L \subset N,
\ee
and
\be
R = \{\alpha \in N | C_N(\alpha,N)=1\},
\ee
where
\be
C_N(\alpha,\beta) = \eta^{\sum_{j=0}^{k-1}\left< j \nu^j \alpha, \beta\right> },
\ee
and noting that $M \subset R$. If $Q$ is a subgroup of $L$, we denote by $\hat{Q}$ the subgroup of $L_\nu$ obtained by pulling back $Q$. By Proposition 6.1 in \cite{L1}, there is a unique homomorphism $\tau: \hat{M} \rightarrow \mathbb{C}$ satisfying
\be
\tau(\eta) = \eta \mbox{ and } \tau(a\hat{\nu}a^{-1}) = \eta^{-\sum_{j=0}^{k-1} \left< \nu^j \overline{a},\overline{a}\right>} 
\ee
for $a \in \hat{L}_{\nu}$.
We recall the following proposition from \cite{L1}:
\begin{prop}[Proposition 6.2 of \cite{L1}]
There are exactly $|R/M|$ extentions of $\tau$ to a homomorphism $\chi:\hat{R} \rightarrow \mathbb{C}$. For each $\chi$, there is a unique (up to equivalence) irreducible $\hat{N}$-module on which $\hat{R}$ acts according to $\tau$, and every irreducible $\hat{N}$-modules on which $\hat{M}$ acts according to $\tau$ is equivalent to one of these. Every such module has dimension $|N/R|^2$.
\end{prop}

We now present a technical result involving the structure of the lattice that ensures that $N=M=R$. We underscore the fact that the mild  assumptions on the lattice in this result are essential a twisted version of the assumption made in \cite{P} involving the nonsingularity of the Gram matrix.

\begin{prop}
If the twisted Gram matrix associated to the lattice $L$ and its isometry $\nu$ is invertible, that is 
\be
A_L^{\nu}=\left(\left<\a[i],\a[j]\right>\right),\ee
is invertible, then $N=M=R$.
\end{prop}

\begin{proof}
By the definition of these subsets we have $M\subset N$ and $R\subset N$. As such, we will focus on the opposite containment starting with the subset $M$. Suppose $\a\in N$ and write 
$$\a=\sum_{r=1}^d\sum_{s=1}^{\ell_r}m_s^{(r)}\a_s^{(r)}.$$
Now the condition that $\a\in N$ becomes 
\be\begin{aligned}
0=\left<\a,\a[j]\right>&=\sum_{i=1}^d\sum_{s=1}^{\ell_i}m_s^{(i)}\left<\a_s^{(i)},\a[j]\right>\\&=\sum_{i=1}^d\sum_{s=1}^{\ell_i}m_s^{(i)}\left<\nu^{\ell_i-s+1}\a_s^{(i)},\nu^{\ell_i-s+1}\a[j]\right>\\
&=\sum_{i=1}^d\sum_{s=1}^{\ell_i}m_s^{(i)}\left<\a_1^{(i)},\a[j]\right>\\
&=\sum_{i=1}^d\frac{1}{\ell_i}\sum_{s=1}^{\ell_i}m_s^{(i)}\left<\a[i],\a[j]\right>\\
&=\sum_{i=1}^d\left<\a[i],\a[j]\right>\left(\sum_{s=1}^{l_i}\frac{m_s^{(i)}}{\ell_i}\right),\end{aligned}\ee
for all $1\leq j \leq d$. It follows that from the invertibility of $A_L^{\nu}$ that for all $1\leq i \leq d$ we have 
$$\sum_{s=1}^{l_i}m_s^{(i)}=0.$$
It follows that $\a$ is an integer linear combination of terms of the form $\a^{(r)}_{s}-\a^{(r)}_{s+1}=(1-\nu)\a^{(r)}_{s}$ for $1\leq r\leq d$ and $1\leq s\leq \ell_r-1$, and thus $\a\in M$. 

Next we show that $N=R$. By the above argument this will follow if we show that $\a^{(r)}_{s}-\a^{(r)}_{s+1}\in R$ for $1\leq r\leq d$ and $1\leq s\leq \ell_r-1$. Since $k=2\text{lcm}\{\ell_1,\cdots,\ell_d\}$, we can write $k=m\ell_r$ for some $m\in \mathbb{N}$. We begin by observing the behavior of the following two collapsing sums
\be\label{sum1}
\sum_{j=1}^{\ell_r-1}\left<\nu^j\left(\a_s^{(r)}-\a_{s+1}^{(r)}\right),N\right>=0\ee
and 
\be\label{sum2}\begin{aligned}
\sum_{j=1}^{\ell_r-1}j\left<\nu^j\left(\a_s^{(r)}-\a_{s+1}^{(r)}\right),N\right>&=\sum_{j=1}^{\ell_r-1}j\left<\a_{s+j(\text{mod }\ell_r)}^{(r)}-\a_{s+1+j(\text{mod }\ell_r)}^{(r)},N\right>\\&=-\ell_r\left<\a_s^{(r)},N\right>=.\end{aligned}\ee
Now we observe that 
\be\begin{aligned}
\sum_{j=0}^k j\left<\nu^j\left(\a_s^{(r)}-\a_{s+1}^{(r)}\right),N\right>&=\sum_{u=1}^{m-1}\sum_{j=u\ell_r}^{(u+1)\ell_r-1}j\left<\nu^j\left(\a_s^{(r)}-\a_{s+1}^{(r)}\right),N\right>\\
&=\sum_{u=0}^{m-1}\sum_{j=0}^{\ell_r-1}(j+u\ell_r)\left<\nu^{j+u\ell_r}\left(\a_s^{(r)}-\a_{s+1}^{(r)}\right),N\right>\\
&=\sum_{u=0}^{m-1}\sum_{j=0}^{\ell_r-1}(j+u\ell_r)\left<\nu^{j}\left(\a_s^{(r)}-\a_{s+1}^{(r)}\right),N\right>\\
&=-k\left<\a_s^{(r)},N\right>,
\end{aligned}\ee
where in the last step we have used \eqref{sum1} and \eqref{sum2}. This implies that for all $\a\in N$, $C_N(\a,N)=1$ and thus $N=R$.
\end{proof}

In light of this proposition, we assume that $N=M=R$ throughout this work.
\begin{rem}
The assumption that $N=M=R$ is quite natural, and the cases studied in \cite{CalLM4}, \cite{CalMPe}, and \cite{PS1}-\cite{PS2} all satisfy this property. At the end of this work, we present interesting examples in our present setting which also meet this requirement.
\end{rem}
We let $T = \mathbb{C}_\tau$ be the unique one-dimensional irreducible $\hat{N}$-module with character $\tau$. We skip some details of the construction (found in \cite{CalLM4} and \cite{L1}) but recall that
\be
V_L^T\cong S\left(\hh[\nu]^{-}\right)\otimes U_T
\ee
where $U_T\cong\CC[\hat{L}_\nu]\otimes_{\CC[\hat{N}]} T \cong \CC[L/N]$, on which $\hat{L}_{\nu}$, $\hh[\nu]_{\frac{1}{k}\ZZ}$, $\h_{(0)}$, and $x^h$ for $h\in\h$ act as described in \cite{CalLM4}. We may therefore write
\be
V_L^T \cong S\left(\hh[\nu]^{-}\right)\otimes \CC[L/N].
\ee

We now recall the necessary twisted vertex operators for use in our purposes - examining the principal subspace of $V_L^T$.
As in \cite{L1} and \cite{CalLM4}, we define the formal Laurent series:
\begin{equation}
E^{\pm}(-\alpha,x)=\text{exp}\left(\sum_{n\in\pm\frac{1}{k}\mathbb{Z}_{+}}\frac{-\alpha_{(kn)}(n)}{n}x^{-n}\right).
\end{equation}
 We begin by recalling the following relationship between the series $E^{\pm}(\a,x)\in\left(\text{End }V_L\right)[[x,x^{-1}]]$ (see \cite{L1}):
\be
E^{+}(\a,x_1)E^{-}(\b,x_2)=E^{-}(\b,x_2)E^{+}(\a,x_1)\prod_{j\in\ZZ/k\ZZ}\left(1-\eta^{j}\frac{x_2^{1/k}}{x_1^{1/k}}\right)^{\<\nu^{j}\a,\b\>}
\ee
for $\a,\b\in L$. Let
\be\label{constant}
\sigma(\a)=\prod_{0<j<\frac{k}{2}}(1-\eta^{-j})^{\<\nu^{j}\a,\a\>}2^{\<\nu^{k/2}\a,\a\>/2}.
\ee

Following \cite{CalLM4} and using (\ref{exp}), for $a \in \hat{L}$ define the $\hat{\nu}$-twisted vertex operator
\begin{equation} 
Y^{\hat{\nu}}(\iota(a), x)= k^{-\left< \overline{a},
\overline{a}\right> /2} 
\sigma(\overline{a}) E^{-} (-\overline{a}, x) E^{+} (-\overline{a}, x)
a x^{\overline{a}_{(0)}+\left<
\overline{a}_{(0)},
\overline{a}_{(0)}\right> /2-\left< \overline{a},
\overline{a}\right> /2}.
\end{equation}
In particular, for $\alpha \in L$, we have
\begin{equation} 
Y^{\hat{\nu}}(\iota(e_\alpha), x)= k^{-\left< \a,
\a\right> /2} 
\sigma(\a) E^{-} (-\a, x) E^{+} (-\a, x)
e_\alpha x^{\a_{(0)}+\left<
\a_{(0)},
\a_{(0)}\right> /2-\left< \a,
\a\right> /2},
\end{equation}
for $\a\in\h$. For $m\in\frac{1}{k}\ZZ$ and $\a\in L$ define the component operators $\left(e^{\a}\right)^{\hnu}_m$ by 
\begin{equation}\label{VertexOperators}
Y^{\hnu}(\iota(e_{\alpha}),x)=\sum_{m\in\frac{1}{k}\ZZ}\left(e^{\a}\right)^{\hnu}_mx^{-m-\frac{\left<\alpha,\alpha\right>}{2}}.
\end{equation}
As in \cite{CalLM4}, we also extend our twisted vertex operators to $V_L$, so that
\be
Y^{\hnu}(v,x)=\sum_{m\in\frac{1}{k}\ZZ}v_m x^{-m-1}
\ee
(we refer the reader to \cite{DL1} and \cite{CalLM4} for details regarding the construction of $Y^{\hnu}$).
Using the twisted vertex operators $Y^{\hnu}$, we have that $V_L^T$ is a $\hnu$-twisted module for $V_L$, and in particular it satisfies the twisted Jacobi identity:
\begin{multline} \label{Jacobi}
x^{-1}_0\delta\left(\frac{x_1-x_2}{x_0}\right)
Y^{\hat{\nu}}(u,x_1)Y^{\hat{\nu}}(v,x_2)-x^{-1}_0
\delta\left(\frac{x_2-x_1}{-x_0}\right) Y^{\hat{\nu}}
(v,x_2)Y^{\hat{\nu}}_T (u,x_1)\\  
= x_2^{-1}\frac{1}{k}\sum_{j\in \Z /k \Z}
\delta\left(\eta^j\frac{(x_1-x_0)^{1/k}}{x_2^{1/k}}\right)Y^{\hat{\nu}} 
(Y(\hat{\nu}^j
u,x_0)v,x_2) 
\end{multline}
for $u, v \in V_L$.

\section{Principal Subspaces}
In this section and for the remainder of this work, we assume that $\left<\a_i,\a_j\right>\geq 0$ for all $1\leq i,j\leq D$. For the lattice vertex operator algebra $V_{L}$ we define the principal subalgebra of $V_L$ corresponding to the choice of $\ZZ$-basis $\mathcal{B}=\{\a_1,\dots,\a_D\}$ of $L$ by 
\be 
W_L(\mathcal{B})=\left<e^{\a_1},\dots, e^{\a_D}\right>,\ee
where by $\left<e^{\a_1},\dots, e^{\a_D}\right>$, we mean the smallest vertex subalgebra of $V_L$ containing all generating vectors. As $\mathcal{B}$ will be fixed we write $W_L=W_L(\mathcal{B})$. We recall the following results that describe $W_L$:

\begin{thm}\label{untwistedpresentation} \cite{P}
 Let $V_L$ be the lattice vertex algebra constructed from a rank $n$ integral lattice $L$ with the condition that $\left<\alpha_i,\alpha_j\right>\geq0$ and $\left<\a_i,\a_i\right>\in 2\mathbb{Z}$. Let 
\begin{equation*}
U_L:=\mathbb{C}[x_{i}(m)|1\leq i \leq n,m<0]
\end{equation*}
 Consider an ideal in $U_L$ 
 \begin{equation*}
I_L=\sum_{i=1}^{n}\sum_{j=1}^{n}\sum_{k=1}^{\left<\alpha_i,\alpha_j\right>}\sum_{l\leq 0}{U_L\cdot R^{(i,j)}_{k,l}},
\end{equation*}where the $R$'s are certain quadratic expressions in $U_L$.  Then we have
\begin{equation*}
W_L\cong  U_L/I_L.
\end{equation*}
\end{thm}

 The aim of this work is to extend Theorem \ref{untwistedpresentation} to our twisted setting. We introduce the notion of the principal subspace of a twisted module $V_L^T$ in general. Let $1_T\in V_L^T$ be a highest weight vector of $V_L^T$ and define the principal subspace of $V_L^T$ corresponding to $\mathcal{B}$ by
\be
W_L^T(\mathcal{B})=W_L(\mathcal{B})\cdot 1_T,\ee
again, as the $\ZZ$-basis $\mathcal{B}$ will be fixed we write $W_L^T=W_L^T(\mathcal{B})$.~\\

\subsection{Grading of $W_L^T$}~

Recall from \cite{L1} and\cite{CalLM4} that $V_L^T$, and thus $W_L^T$ is $\frac{1}{k}\ZZ$-graded by the eigenvalues of $L^{\hnu}(0)$, with
\be
L^{\hnu}(0)1_T=\frac{1}{4k^2}\sum_{j=1}^{k-1}j(k-j)\text{dim }\h_{(j)}1_T\ee
and so we write
\be
\text{wt}(1_T)=\frac{1}{4k^2}\sum_{j=1}^{k-1}j(k-j)\text{dim }\h_{(j)}.\ee
We now determine the $L^{\hnu}(0)$-weight of an arbitrary monomial in $W_L^T$, which takes the form
\be\label{generalmonomial}
\ea{i_1}{m_1}\cdots\ea{i_r}{m_r}\cdot 1_T\in W_L^T
\ee
where $1\leq i_j\leq D$ and $m_j\in\frac{1}{k}\ZZ$. It follows from Proposition 6.3 of \cite{DL1} that 
\begin{multline}\label{YomegaCommutator}
[Y^{\hnu}(\omega,x_1),Y^{\hnu}(\iota(e_\alpha),x_2)]\\
 = x_2^{-1}\frac{d}{dx_2}Y^{\hnu}(\iota(e_\alpha),x_2)\delta(x_1/x_2) - \frac{1}{2}\left< \alpha,\alpha \right> x_2^{-1} Y^{\hnu}(\iota(e_\alpha),x_2)\frac{d}{dx_1}\delta(x_1/x_2),
\end{multline}
where $\delta(x) = \sum_{n \in \mathbb{Z}}x^n$. Taking $\mathrm{Res}_{x_2} \mathrm{Res}_{x_1}$ of $x_1 x_2^m$ times of both sides of (\ref{YomegaCommutator}) immediately gives
\begin{equation}
[L^{\hnu}(0),\ea{}{m}]=-m\ea{}{m}.
\end{equation}
Applying this to (\ref{generalmonomial}) we have 
\be\begin{aligned}
L^{\hnu}(0)&\ea{i_1}{m_1}\cdots\ea{i_r}{m_r}\cdot 1_T\\&=\left(-(m_1+\cdots+m_r)+\text{wt}(1_T)\right)\ea{i_1}{m_1}\cdots\ea{i_r}{m_r}\cdot 1_T\end{aligned}\ee
which implies that 
\be
\text{wt}(\ea{i_1}{m_1}\cdots\ea{i_r}{m_r}\cdot 1_T)=-(m_1+\cdots+m_r)+\text{wt}(1_T).\ee

We also endow $W_L^T$ with $d$-additional gradings, which together will form a $d$-tuple known as the {\em charge}, and whose sum forms what we call the {\em total charge}. First, let
\be
\mathcal{B}^{\star}=\{\lambda_i\in\mathbb{Q}\otimes_{\Z}L|1\leq i\leq D \text{ and } \left<\a_i,\lambda_j\right>=\delta_{i,j}\}\ee
and 
\be
\left(\lambda_{i}\right)_{(0)} =\frac{1}{l_r}(\lambda_{i}+\nu \lambda_i+\cdots+\nu^{l_r-1}\lambda_i),\ee
where $\a_i=\a^{(r)}_j$ for some $1\leq r\leq d$ and $1\leq j\leq l_r$.  For notational convenience, we set $\a^{(j)}:=\a^{(j)}_1$ (and similarly define $\lambda^{(j)}$) for each $j=1,\dots,d$, a fixed representative from the $\nu$-orbit of $\a^{(j)}_1$.
Let $\lambda^{(r)}$ be the element of $\mathcal{B}^{\star}$ corresponding to $\a^{(r)}$ and set
\be
\mathfrak{o}_r=\{\lambda\in\mathcal{B}^{\star}|\lambda=\nu^{m}\lambda^{(r)} \text{ for some } m\in\ZZ\}\ee
and notice that these sets partition $\mathcal{B}^{\star}$ as $r$ runs from 1 to $d$. We now consider the orbit sum
\be 
\lambda[r]=\sum_{\lambda\in\mathfrak{o}_r}(\lambda)_{(0)},\ee
and define the charge grading 
\be\begin{aligned}
\text{ch}(\ea{}{m})&=\left(\left<\a,\lambda[1]\right>,\dots,\left<\a,\lambda[d]\right>\right)\\
&=\left(l_1\left<\a,\left(\lambda^{(1)}\right)_{(0)}\right>,\dots,l_d\left<\a,\left(\lambda^{(d)}\right)_{(0)}\right>\right) \in \mathbb{Z}^d.\end{aligned}\ee
\begin{rem}
The charge grading we use here is analogous to the charge grading used in \cite{CLM1}-\cite{CLM2} and \cite{CalLM1}-\cite{CalLM3} in the untwisted setting, and to the charge gradings used in \cite{CalLM4}, \cite{CalMPe}, and \cite{PS1}-\cite{PS2}. The multiplication by $l_i$ in each component ensures that the $(\lambda^{(i)})_{(0)}$-charge of each element is an integer.
\end{rem}

Now that we have endowed $W_L^T$ with ($d+1$)-gradings, define the homogeneous graded components

\be
\left(W_L^T\right)_{(n,\mathbf{m})} =\{v\in W_L^T| \text{wt }v=n, \text{ch }v=\mathbf{m}\}.\ee
 and the multigraded dimension
\be
\chi(q;\mathbf{x})= \mathrm{tr}|_{W_L^T}x_1^{l_1(\lambda^{(1)})_{(0)}}\cdots x_d^{l_d(\lambda^{(d)})_{(0)}} q^{k\hat{L}^{\hnu}(0)},\ee
where $\mathbf{x}^{\mathbf{m}}=x_1^{m_1}\cdots x_d^{m_d}$. We also define the shifted multigraded dimensions
\be
\chi'(q;\mathbf{x})= q^{-\mathrm{wt}(1_T)}\chi(q;\mathbf{x}) = \sum_{\substack{n\in\mathbb{Z}_{\ge 0} \\ \mathbf{m}\in(\mathbb{Z}_{\ge 0})^d}} \text{dim }\left(W_L^T\right)_{(n,\mathbf{m})} q^n \mathbf{x}^{\mathbf{m}}\ee
so that the powers of $x_1,\dots ,x_d$ and $q$ are all integers.~\\

\subsection{The Structure of $W_L^T$}~

In this subsection we investigate the structure of $W_L^T$ that will be exploited in the following section to provide a presentation for $W_L^T$. From \cite{L1} and \cite{CalLM4}, we have that
\be\label{limitrelation}
Y^{\hnu}(\hnu^{r}v,x)=\lim_{x^{1/k}\to \eta^{-r}x^{1/k}}Y^{\hnu}(v, x)\ee
and so, to investigate operators of the form $(e^{\a_i})^{\hnu}_{n}$ for $1\leq i\leq D$ and $n\in\frac{1}{k}\ZZ$ we need only consider the vertex operators $Y^{\hnu}(e^{\a^{(j)}}, x)$ for $1\leq j\leq d$. In fact, if $l_j$ is odd or if $l_j$ is even and $\left<\a^{(j)},\nu^{\frac{l_j}{2}}\a^{(j)}\right>\in 2\mathbb{Z}$ we have 
\be\begin{aligned}\label{vertexoperatoreven}
Y^{\hnu}(e^{\a^{(j)}}, x)&=\sum_{n\in\frac{1}{l_j}\ZZ}(e^{\a^{(j)}})^{\hnu}_{n} x^{-n-\frac{<\a^{(j)},\a^{(j)}>}{2}}\\& \in (\text{End }V_L^T)[[x^{1/{l_{j}}},x^{-1/{l_{j}}}]]\subset (\text{End }V_L^T)[[x^{1/k},x^{-1/k}]]\end{aligned}.\ee

Further, if $l_j$ is even and $\left<\a^{(j)},\nu^{\frac{l_j}{2}}\a^{(j)}\right>\notin 2\mathbb{Z}$ we have

\be\begin{aligned}\label{vertexoperatorodd}
Y^{\hnu}(e^{\a^{(j)}}, x)&=\sum_{n\in\frac{1}{2l_j} + \frac{1}{l_j}\ZZ}(e^{\a^{(j)}})^{\hnu}_{n} x^{-n-\frac{<\a^{(j)},\a^{(j)}>}{2}}\\& \in (\text{End }V_L^T)[[x^{1/{2l_{j}}},x^{-1/{2l_{j}}}]]\subset (\text{End }V_L^T)[[x^{1/k},x^{-1/k}]].\end{aligned}\ee

Now applying (\ref{limitrelation}) to (\ref{vertexoperatoreven}) and (\ref{vertexoperatorodd}) we have the following lemmas.

\begin{lem}\label{limitlemma1}
We have 
\be(e^{\nu^r\a^{(i)}})^{\hnu}_n=\eta_{l_i}^{rnl_i}(e^{\a^{(i)}})^{\hnu}_n,\ee
if $\a^{(i)}$ satisfies the evenness condition and 
\be(e^{\nu^r\a^{(i)}})^{\hnu}_n=\eta_{2l_i}^{2rnl_i-1}(e^{\a^{(i)}})^{\hnu}_n,\ee
otherwise.
\end{lem}

Observe that since $\nu$ is an isometry we have $\a_j^{(r)}$ satisfies the evenness condition if and only if $\a^{(r)}$ does.
Define the following subsets of $\frac{1}{k}\ZZ$ for each $i \in \mathbb{Z}$ such that $1 \le i \le d$:
\be
Z_i=\begin{cases}\frac{1}{l_i}\mathbb{Z} & \text{ if } \a^{(i)} \text{ satisfies the evenness condition,}\\
\frac{1}{2l_i}+\frac{1}{l_i}\ZZ & \text{ otherwise.}\end{cases}\ee
We also set
\be 
L_i=\begin{cases}l_i& \text{ if } \a^{(i)} \text{ satisfies the evenness condition,}\\
2l_i & \text{ otherwise.}\end{cases}\ee

For any $\a,\b\in L$ using (\ref{VertexOperators}) we have 
\be
Y^{\hnu}(e^{\alpha},x)e^{\b}=k^{-\left<\a,\a\right>/2}\epsilon(\a,\b)\sigma(\a)x^{\left<\a_{(0)},\b\right>+\frac{\left<\a_{(0)},\a_{(0)}\right>}{2}-\frac{\left<\a,\a\right>}{2}}E^{-}(-\a,x)e^{\a+\b},
\ee
in which the smallest power of $x$ is $\left<\a_{(0)},\b\right>+\frac{\left<\a_{(0)},\a_{(0)}\right>}{2}-\frac{\left<\a,\a\right>}{2}$, which implies that 
\be
\left(e^{\a}\right)^{\hnu}_n e^{\b}=0\ee
for all $n>-\left<\a_{(0)},\b\right>-\frac{\left<\a_{(0)},\a_{(0)}\right>}{2}$. Specializing this to the case when $\b=0$ we have 
\be\label{zeroops}
\left(e^{\a}\right)^{\hnu}_n 1_T=0\ee
for all $n>-\frac{\left<\a_{(0)},\a_{(0)}\right>}{2}$ and 
\be\label{highestop}
\left(e^{\a}\right)^{\hnu}_{-\frac{\left<\a_{(0)},\a_{(0)}\right>}{2}}1_T=k^{\left<\alpha,\alpha\right>/2}\sigma(\a)e^{\a}.\ee
In light of (\ref{zeroops}) and (\ref{highestop}) we define the subsets of $Z_i$
\be
Z_i^{-}=\left\{n\in Z_i\left| n\leq -\frac{\left<\a^{(i)}_{(0)},\a^{(i)}_{(0)}\right>}{2}\right.\right\}.\ee
and 
\be
Z_i^{+}=\left\{n\in Z_i\left| n> -\frac{\left<\a^{(i)}_{(0)},\a^{(i)}_{(0)}\right>}{2}\right.\right\},\ee
so that we have 
\be\label{zeros}\left(e^{\a^{(i)}}\right)^{\hnu}_n 1_T=0 \text{ for all } n\in Z_i^{+}.\ee

For each pair of $\a_i,\a_j$ in the $\ZZ$-basis of $L$ we set 
\be\label{Nmin}
N_{i,j}=\text{max}\left(\{0\}\cup \{-\left<\nu^r\a_i,\a_j\right>|r\in\ZZ\}\right)\ee
and observe that for $n\geq N_{i,j}$ and $r\in \ZZ$ we have 
\be
(e^{\nu^r\a_i})_ne^{\a_j}=0.\ee
Applying this fact to the appropriate coefficient of $x_0$ in the twisted Jacobi identity (\ref{Jacobi}),
 we have $N_{i,j}$ is the smallest non-negative integer such that
\be
(x_1-x_2)^{N_{i,j}}[Y^{\hnu}(e^{\a_i},x_1),Y^{\hnu}(e^{\a_j},x_2)]=0.\ee
This is the weak commutativity property of twisted vertex operators as constructed in \cite{DLeM}. In this work we have $N_{i,j}=0$ for all $1\leq i,j\leq D$ and so the associated vertex operators always commute. Also in this case, by Theorem 3.9 in \cite{DLeM} we have
\be
Y^{\hnu}(Y(e^{\a_i},x_0)e^{\a_j},x_2)=Y^{\hnu}(e^{\a_i},x_2+x_0)Y^{\hnu}(e^{\a_j},x_2).\ee
Extracting the coefficient of $x_0^{n-1}$ while recalling the binomial expansion convention 
\be
(x_2+x_0)^r=\sum_{m\geq0}\binom{r}{m}x_2^{r-m}x_0^m\ee
yields 
\be 
Y^{\hnu}((e^{\a_i})_{-m}e^{\a_j},x)=\frac{1}{(m-1)!}\left(\frac{\partial}{\partial x}\right)^{m-1}\bigg(Y^{\hnu}(e^{\a_i}, x)\bigg)Y^{\hnu}(e^{\a_j},x).\ee
From the theory of untwisted vertex operators \cite{LL} we have 
\be 
(e^{\a_i})_{-m}e^{\a_j}=0
\ee
for all $1\leq m\leq \left<\a_i,\a_j\right>$, which implies that for such $m$ we have 
\be\label{derivativerelations}
\frac{1}{(m-1)!}\left(\frac{\partial}{\partial x}\right)^{m-1}\bigg(Y^{\hnu}(e^{\a_i}, x)\bigg)Y^{\hnu}(e^{\a_j},x)=0.\ee
It follows that for all $1\leq m\leq \left<\nu^r \a^{(i)},\a^{(j)}\right>$ we have 
\be\label{derivativerelations2}
\frac{1}{(m-1)!}\left(\frac{\partial}{\partial x}\right)^{m-1}\bigg(Y^{\hnu}(e^{\nu^{r}\a^{(i)}}, x)\bigg)Y^{\hnu}(e^{\a^{(j)}},x)=0.\ee

Extracting appropriate coefficients of the formal variable $x$ from (\ref{derivativerelations2}) and applying the results of Lemma \ref{limitlemma1} we have the expressions

\be
R(i,j,r,m|t)=\sum_{\substack{n_1+n_2=-t \\ n_1\in Z_i^{-}\\n_2\in Z_j^{-}}}\eta_{L_i}^{rn_1L_i}\binom{-n_1-\frac{\left< \a^{(i)}, \a^{(i)} \right>}{2}}{m-1}(e^{\a^{(i)}})^{\hnu}_{n_1}(e^{\a^{(j)}})^{\hnu}_{n_2},\ee

which have the property that 
\be
R(i,j,r,m|t) v=0\ee
for all $v\in V_L^T$. Observe this gives us $l_i\left<\left(\a^{(i)}\right)_{(0)},\a^{(j)}\right>$ total relations.

\section{A presentation of $W_L^T$}

In this section we provide the necessary construction for, and a proof of our main result - a presentation of the twisted principal subspace $W_L^T$. Recall the form of a general monomial (\ref{generalmonomial}) and notice that by (\ref{Jacobi}) our monomial may take the form
\be
(e^{\b_1})^{\hnu}_{m_1}\cdots (e^{\b_r})^{\hnu}_{m_r}\cdot 1_T\ee
where $\beta_i\in\{\a^{(1)},\dots,\a^{(d)}\}$ and $m_i\in\frac{1}{k}\ZZ_{<0}$.~\\

\subsection{A polynomial ring lifting of the principal subspace} ~\\

We now recall from \cite{P} the universal commutative vertex algebra $U_L$ on generators $x_{\a_i}=x_{\a_i}(-1)$ for $1\leq i\leq D$ where the vertex operators are given by  
\be
Y_{U_L}(x_{\a_i}(-1),x)=\sum_{n\in\ZZ_{<0}} x_{\a_i}(n)x^{-n-1}\ee
subject to the commutation relations 
\be 
[Y_{U_L}(x_{\a_i}(-1),x_1),Y_{U_L}(x_{\a_j}(-1),x_2)]=0
\ee 
for all $x_{\a_i}(-1),x_{\a_j}(-1)\in U_L$ and no other relations. This object was used in \cite{P} to study the untwisted version of the work in this paper. Explicitly, we can realize $U_L$ as the principal subalgebra of the rank $D$ lattice vertex operator algebra whose associated Gram matrix is the $D\times D$ zero matrix. Observe that we may view $U_L$ as a polynomial algebra in infinitely many variables
\be
U_L\cong\CC[x_{\a_i}(n)|1\leq i\leq D, n\leq -1].\ee

There is a natural surjection defined on monomials by
\be\begin{aligned}\label{untwistedsurjection}
f_L:U_L&\to W_L\\
x_{\a_{i_1}}(n_{i_1})\cdots x_{\a_{i_j}}(n_{i_j})&\mapsto (e^{\a_{i_1}})_{n_1}\cdots (e^{\a_{i_j}})_{n_j}\cdot 1
\end{aligned}\ee
and extended linearly.
Under this surjection, we may lift the automorphism $\hnu$ restricted to $W_L$ to an automorphism of $U_L$, denoted $\tilde{\nu}$ so that the following diagram commutes.
\be
\begin{tikzcd}
U_L \arrow{r}{\tilde{\nu}} \arrow[swap]{d}{f_L} & U_L \arrow{d}{f_L} \\
W_L  \arrow{r}{\hnu} & W_L
\end{tikzcd}
\ee

Using $U_L$ as motivation, we define
\be
U_L^T=\CC\bigg[ x_{\a^{(i)}}^{\tnu}(n)\bigg| 1\leq i\leq d, n\in Z_i \bigg].\ee
For notational convenience, we define 
\be\label{ExtraTerms}
x_{\a^{(i)}}^{\tnu}(n)=0\mbox{ whenever }n \notin Z_i.
\ee
We have a surjection (\ref{untwistedsurjection}) defined by
\be\begin{aligned}
f_L^T:U_L^T&\to W_L^T\\
x_{\a^{(i_1)}}^{\tnu}(n_{i_1}) \cdots x_{\a^{(i_j)}}^{\tnu}(n_{i_j})&\mapsto (e^{\a^{(i_1)}})^{\hnu}_{n_{i_1}}\cdots  (e^{\a^{(i_j)}})^{\hnu}_{n_{i_j}} \cdot 1_T,\end{aligned}\ee
and extended linearly to all of $U_L^T$. We note that this map is well-defined since the operators  $(e^{\a^{(i)}})^{\hnu}_n$ all commute.

\begin{rem}
We note here that $U_L^T$ plays the same role that the universal enveloping algebra $U(\on)$ played in in \cite{PS1}-\cite{PS2}, \cite{CalMPe}, and \cite{CalLM4}.  In \cite{P}, $U_L$ played similar in place of the universal enveloping algebra $U(\bar{\frak{n}})$ from \cite{CLM1}-\cite{CLM2} and \cite{CalLM1}-\cite{CalLM3}.  
\end{rem}

The aim for the remaining portion of this section is to describe the ideal $\text{Ker }f_L^T\subset U_L^T$ in terms of certain quadratic expressions. In this commutative setting we need only consider expressions that are constructed from (\ref{derivativerelations}) and lifted to $U_L^T$.
With this, we consider the following expressions 
\be
R(i,j,r,m|t)=\sum_{\substack{n_1+n_2=-t \\ n_1\in Z_i^{-}\\n_2\in Z_j^{-}}}\eta_{L_i}^{rn_1L_i}\binom{-n_1-\frac{\left< \a^{(i)}, \a^{(i)} \right>}{2}}{m-1}x_{\a^{(i)}}^{\tnu}(n_1)x_{\a^{(j)}}^{\tnu}({n_2}),\ee
 where $-t \in Z_i^- + Z_j^-$. Let $J_L^T\subset U_L^T$ be the left ideal generated by all of these sums. Also, define

\be U_L^{T+}=U_L^T \CC\bigg[ x_{\a^{(i)}}^{\tnu}(n)\bigg| 1\leq i\leq d, n\in Z_i^+ \bigg].\ee

By (\ref{zeros}), we have 
\be
f_L^T(x_{\a^{(i)}}^{\tnu}(n)) = 0
\ee
for all $n \in Z^+_i$. We define the ideal $I_L^T$ by
\be
I_L^T = J_L^T + U_L^{T+}
\ee 
and will now proceed to show that 
Ker$f_L^T = I_L^T$.~\\

\begin{rem}
In this setting we are defining $U_L^{T+}$ slightly differently than in previous works, where it was called $U(\bar{\frak{n}})\bar{\frak{n}}_+$. Observe that, for example, it contains the monomials 
\be
x_{\a^{(i)}}^{\tnu}\left(-\frac{1}{L_i}\right),\dots, x_{\a^{(i)}}^{\tnu}\left(-\frac{\left<\a^{(i)}_{(0)},\a^{(i)}_{(0)}\right>}{2}+\frac{1}{l_i}\right),\ee
for all $1\leq i\leq d$. This is in parallel to the terms that must be added to the ideals in the case of principal subspaces of non-vacuum modules for lattice vertex operator algebras. In \cite{P}, we found that if a lattice had an integral basis given by $\{\b_1,\dots,\b_d\}$ with nondegenerate bilinear form $\left<\cdot,\cdot\right>:L\times L\to \mathbb{Z}$ such that $\left<\b_i,\b_j\right>$ and $\{\omega_1,\dots,\omega_d\}$ is the associated dual basis then 
\be
W_{L+\omega_i}=U_L/I_{L+\omega_i}\ee
where
\be 
I_{L+\omega_i}=J_L+U_L^++U_Lx_{\a_i}(-1).
\ee\\
\end{rem}

\subsection{New Relations}~\\
We begin by identifying several important elements of $I_L^T$ which will be needed later.
\begin{lem}\label{NewRelations}
For all $i,j,s,t \in \mathbb{Z}$ such that $1\le i,j \le d$,  $s,t\ge 0$, and $s+t \le l_i \left< \alpha^{(i)}_{(0)}, \alpha^{(j)} \right>-1$, we have
\begin{equation}
x^{\hnu}_{\alpha^{(i)}}\left({-\frac{\left< \alpha^{(i)}_{(0)},\alpha^{(i)}_{(0)} \right>}{2}-\frac{s}{l_i}}\right) x^{\hnu}_{\alpha^{(j)}}\left({-\frac{\left< \alpha^{(j)}_{(0)},\alpha^{(j)}_{(0)} \right>}{2}-\frac{t}{l_i}}\right) \in I_L^T.
\end{equation}
\end{lem}

\begin{proof}
For notational convenience in this proof, we write
\be
a_i = \frac{\left< \alpha^{(i)}_{(0)},\alpha^{(i)}_{(0)} \right>}{2}
\ee
and
\be
a_j = \frac{\left< \alpha^{(j)}_{(0)},\alpha^{(j)}_{(0)} \right>}{2}
\ee
Define $q = l_i \left< \alpha^{(i)}_{(0)}, \alpha^{(j)} \right>-1-s-t$. Consider the relations
\begin{equation}
R\left(i,j,r,m \bigg\vert a_i+a_j+\frac{s+t}{l_i}\right)=\sum_{\substack{n_1+n_2=-\left(a_i+a_j+\frac{s+t}{l_i}\right) \\ n_1\in Z_i^{-}\\n_2\in Z_j^{-}}}\eta_{L_i}^{rn_1L_i}\binom{-n_1-\frac{\left< \a^{(i)}, \a^{(i)} \right>}{2}}{m-1}x_{\a^{(i)}}^{\tnu}(n_1)x_{\a^{(j)}}^{\tnu}({n_2})
\end{equation}
for $0 \le r \le l_i-1$ and $1 \le m \le \left< \nu^r \alpha^{(i)},\alpha^{(j)} \right>$
and note that each sum contains a multiple of the element
$x^{\hnu}_{\alpha^{(i)}}\left({-a_i-\frac{s}{l_i}}\right) x^{\hnu}_{\alpha^{(j)}}\left({-a_j-\frac{t}{l_i}}\right)$.
To each such sum, we add the $q$ terms
\begin{equation}
\sum_{l=1}^q \eta_{L_i}^{r(-a_i-\frac{s+t+l}{l_i})L_i}\binom{a_i+\frac{s+t+l}{l_i}-\frac{\left< \a^{(i)}, \a^{(i)} \right>}{2}}{m-1}
x^{\hnu}_{\alpha^{(i)}}\left({-a_i-\frac{s+t+l}{l_i}}\right) x^{\hnu}_{\alpha^{(j)}}\left({-a_j+\frac{l}{l_i}}\right),
\end{equation}
which is an element of $U_L^{T+}$, and define the new sum of elements from $I_L^T$:
\begin{eqnarray*}
\lefteqn{R'\left(i,j,r,m\bigg\vert a_i+a_j+\frac{s+t}{l_i}\right)}\\
&&=\sum_{p=0}^{l_i\left< \alpha^{(i)}_{(0)}, \alpha^{(j)} \right> -1 }\eta_{L_i}^{r(-a_i- \frac{p}{l_i})L_i}\binom{a_i + \frac{p}{l_i}-\frac{\left< \a^{(i)}, \a^{(i)} \right>}{2}}{m-1}x_{\a^{(i)}}^{\tnu}\left(-a_i- \frac{p}{l_i}\right)x_{\a^{(j)}}^{\tnu}\left({-a_j-\frac{s+t}{l_i}+\frac{p}{l_i}}\ \right)
\end{eqnarray*}
for $0 \le r \le l_i-1$ and $1 \le m \le \left< \nu^r \alpha^{(i)},\alpha^{(j)} \right>$.
We rewrite each of these new sums as the matrix equation
\be A \mathbf{x}_{i,j}=\mathbf{R}\ee
where
\be
({\mathbf{x}_{i,j}})_p=x_{\a^{(i)}}^{\tnu}\left(-a_i- \frac{p}{l_i}\right)x_{\a^{(j)}}^{\tnu}\left({-a_j-\frac{s+t}{l_i}+\frac{p}{l_i}}\ \right)\ee
$0\leq p \leq l_i\left<\alpha^{(i)}_{(0)},\alpha^{(j)}\right>-1$,
where 
where the $(r,m)^{th}$ row (ordered lexicographically) of $A$ is $\mathcal{R}(r,m)^t\in\CC^{l_i\left<\alpha^{(i)}_{(0)},\alpha^{(j)}\right>}$ has $p$th entry given by 
\be
\mathcal{R}(r,m)_p=\eta_{L_i}^{r(-a_i- \frac{p}{l_i})L_i}\binom{a_i + \frac{p}{l_i}-\frac{\left< \a^{(i)}, \a^{(i)} \right>}{2}}{m-1},\ee
and where
\be \mathbf{R}^t=\left(R'\left(i,j,r,m\bigg\vert a_i+a_j+\frac{s+t}{l_i}\right)\right)    \ee
where  $1\leq m\leq \left<\nu^r\alpha^{(i)},\alpha^{(j)}\right>$, $0\leq r<l_i$ and the entries are arranged to correspond to the structure of $A$. The matrix $A$ is a square matrix of dimension $l_i \left< \alpha^{(i)}_{(0)},\alpha^{(j)}\right>$ which is row-equivalent to the matrix in Theorem \ref{GeneralizedPascal} of the Appendix. Using Theorem \ref{GeneralizedPascal}, we have that $A$ is invertible, and rewriting our matrix equation as
\be  \mathbf{x}_{i,j}=A^{-1}\mathbf{R}\ee
we have that 
\begin{equation}
x^{\hnu}_{\alpha^{(i)}}\left({-a_i-\frac{s}{l_i}}\right) x^{\hnu}_{\alpha^{(j)}}\left({-a_j-\frac{t}{l_i}}\right)
\end{equation}
is a linear combination of elements of $I_L^T$, and so
\be 
x^{\hnu}_{\alpha^{(i)}}\left({-a_i-\frac{s}{l_i}}\right) x^{\hnu}_{\alpha^{(j)}}\left({-a_j-\frac{t}{l_i}}\right) \in I_L^T
\ee
\end{proof}

\begin{rem}
The relations in this subsection did not appear in the earlier works \cite{CalLM4}, \cite{CalMPe}, and \cite{PS1}-\cite{PS2} on principal subspaces of standard modules for twisted affine Lie algebras. Such relations could have certainly been derived using a similar technique, but the relations are trivially true. Relations similar to the ones derived in this subsection, were, however, necessary in \cite{P} and \cite{MPe}, where the more general lattice case was studied. 
\end{rem}~\\

\subsection{Important Morphisms}~\\

In the spirit of \cite{CalLM4},\cite{CalMPe}, \cite{PS1}, and \cite{PS2} we consider the following shifting maps on $U_L^T$. For $\gamma\in\h_{(0)}$ define
\be\begin{aligned}\label{definetau}
\tau_{\gamma}:U_L^T&\to U_L^T\\
x_{\a^{(i)}}^{\tnu}(n)&\mapsto x_{\a^{(i)}}^{\tnu}\left(n+\left<\left(\a^{(i)}\right)_{(0)},\gamma\right>\right),\end{aligned}\ee
where $n\in Z_i$. 

As in previous work \cite{PS1}-\cite{PS2}, we use elements of $\h_{(0)}$ associated to the basis, $\mathcal{B}^{\star}$, of $L^{\circ}$. As such, for $1 \le i \le d$, we set
\be
\gamma_i=\left(\lambda^{(i)}\right)_{(0)}=\frac{1}{l_i}\sum_{r=0}^{l_i-1}\nu^r\lambda^{(i)}.\ee
Observe that we have 
$$\ta{i}\left(x_{\a^{(j)}}^{\tnu}\left(n\right)\right)=x_{\a^{(j)}}^{\tnu}\left(n+\frac{1}{l_i}\delta_{i,j}\right)$$
for each $j \in \mathbb{Z}$ such that $1 \le j \le d$.

\begin{lem} \label{lemma1}For each $i \in \mathbb{Z}$ such that $1 \le i \le d$ we have
$$
\ta{i}\left(I_L^T+U_L^Tx_{\a^{(i)}}^{\tnu}\left(-\frac{\left<\a^{(i)}_{(0)},\a^{(i)}_{(0)}\right>}{2}\right)\right)=I_L^T.$$

\end{lem}

\begin{proof}
We begin by showing 
\be
\ta{i}\left(I_L^T+U_L^Tx_{\a^{(i)}}^{\tnu}\left(-\frac{\left<\a^{(i)}_{(0)},\a^{(i)}_{(0)}\right>}{2}\right)\right)\subset I_L^T.
\ee
First, we not that $\tau_{\gamma_i}(R(j,k,r,m|t) = R(j,k,r,m|t)$ whenever $j \neq i$ and $k \neq i$.
Next, we examine  $\ta{j}\left(R(j,i,r,m|t)\right)$ when $i \neq j$. We have
\begin{eqnarray*}
\lefteqn{\ta{i}\left(R(j,i,r,m|t)\right)}\\
&=&\sum_{\substack{n_1+n_2=-t\\n_1\in Z^-_j, n_2\in Z^-_i }}\eta_{L_j}^{rn_1L_j}\binom{-n_1-\frac{\left< \a^{(j)}, \a^{(j)} \right>}{2}}{m-1}x_{\a^{(j)}}^{\tnu}\left(n_1\right)x_{\a^{(i)}}^{\tnu}(n_2+\frac{1}{l_i})\\
&=&\sum_{\substack{n_1+n_2=-t+\frac{1}{l_i}\\n_1\in Z^-_j, n_2\in Z^-_i}}\eta_{L_j}^{rn_1L_j}\binom{-n_1-\frac{\left< \a^{(j)}, \a^{(j)} \right>}{2}}{m-1}x_{\a^{(j)}}^{\tnu}\left(n_1\right)x_{\a^{(i)}}^{\tnu}(n_2) \\&&+ cx_{\a^{(j)}}^{\tnu}\left(-t+\frac{\left<\a_{(0)}^{(j)},\a_{(0)}^{(i)}\right>}{2}\right)x_{\a^{(i)}}^{\tnu}\left(-\frac{\left<\a_{(0)}^{(i)},\a_{(0)}^{(i)}\right>}{2} + \frac{1}{l_i}\right)\\
&=&R\left(j,i,r,m|t-\frac{1}{l_i}\right) + cx_{\a^{(j)}}^{\tnu}\left(-t+\frac{\left<\a_{(0)}^{(i)},\a_{(0)}^{(i)}\right>}{2}\right)x_{\a^{(i)}}^{\tnu}\left(-\frac{\left<\a_{(0)}^{(i)},\a_{(0)}^{(i)}\right>}{2} + \frac{1}{l_i}\right)
,\end{eqnarray*}
where $c$ is some constant and we note that $$x_{\a^{(j)}}^{\tnu}\left(-t+\frac{\left<\a_{(0)}^{(i)},\a_{(0)}^{(i)}\right>}{2}\right)x_{\a^{(i)}}^{\tnu}\left(-\frac{\left<\a_{(0)}^{(i)},\a_{(0)}^{(i)}\right>}{2} + \frac{1}{l_i}\right)\in U_L^{T+}.$$

In the case that $i \neq j$, we also have
\begin{eqnarray*}
\lefteqn{\ta{i}\left(R(i,j,r,m|t)\right)}\\
&=&\sum_{\substack{n_1+n_2=-t\\n_1\in Z_i^{-}, n_2\in Z_j^{-} }}\eta_{L_i}^{rn_1L_i}\binom{-n_1-\frac{\left< \a^{(i)}, \a^{(i)} \right>}{2}}{m-1}x_{\a^{(i)}}^{\tnu}\left(n_1+\frac{1}{l_i}\right)x_{\a^{(j)}}^{\tnu}(n_2)\\
&=&\eta_{L_i}^{-r\frac{L_i}{l_i}}\sum_{\substack{n_1+n_2=-t+\frac{1}{l_i}\\n_1\in \frac{1}{l_i}+Z_i^-, n_2\in Z_j^- }}\eta_{L_i}^{rn_1L_i}\binom{-n_1-\frac{\left< \a^{(i)}, \a^{(i)} \right>}{2}+\frac{1}{l_i}}{m-1}x_{\a^{(i)}}^{\tnu}\left(n_1\right)x_{\a^{(j)}}^{\tnu}(n_2)\\
  &&+ cx_{\a^{(i)}}^{\tnu}\left(-\frac{\left< \alpha^{(i)}_{(0)},\alpha^{(i)}_{(0)} \right>}{2}+ \frac{1}{l_i}\right)x_{\a^{(j)}}^{\tnu}\left(-t+\frac{\left< \alpha^{(i)}_{(0)},\alpha^{(i)}_{(0)} \right>}{2}\right)\\
&=&\eta_{L_i}^{-r\frac{L_i}{l_i}}\sum_{\substack{n_1+n_2=-t+\frac{1}{l_i}\\n_1\in Z_i, n_2\in Z_j \\n_1,n_2<0}}\eta_{L_i}^{rn_1L_i}\binom{-n_1-\frac{\left< \a^{(i)}, \a^{(i)} \right>}{2}+\frac{1}{l_i}}{m-1}x_{\a^{(i)}}^{\tnu}\left(n_1\right)x_{\a^{(j)}}^{\tnu}(n_2)\\
&& + cx_{\a^{(i)}}^{\tnu}\left(-\frac{\left<\a_{(0)}^{(i)},\a_{(0)}^{(i)}\right>}{2} + \frac{1}{l_i}\right)x_{\a^{(j)}}^{\tnu}\left(-t+\frac{\left<\a_{(0)}^{(i)},\a_{(0)}^{(i)}\right>}{2}\right)
\end{eqnarray*}
for some constant $c$.
We observe that
 \begin{eqnarray}
 \lefteqn{x_{\a^{(i)}}^{\tnu}\left(-\frac{\left<\a_{(0)}^{(i)},\a_{(0)}^{(i)}\right>}{2} + \frac{1}{l_i}\right)x_{\a^{(j)}}^{\tnu}\left(-t+\frac{\left<\a_{(0)}^{(i)},\a_{(0)}^{(i)}\right>}{2}\right)}\\
 &=& x_{\a^{(j)}}^{\tnu}\left(-t+\frac{\left<\a_{(0)}^{(i)},\a_{(0)}^{(i)}\right>}{2}\right)x_{\a^{(i)}}^{\tnu}\left(-\frac{\left<\a_{(0)}^{(i)},\a_{(0)}^{(i)}\right>}{2} + \frac{1}{l_i}\right)\in U_L^{T+}
 \end{eqnarray}
 and only examine the remaining terms.
Using
$$\binom{-n_1-\frac{\left< \a^{(i)}, \a^{(i)} \right>}{2}+\frac{1}{l_i}}{m-1}=\sum_{s=0}^{m-1}\binom{\frac{1}{l_i}}{s}\binom{-n_1-\frac{\left< \a^{(i)}, \a^{(i)} \right>}{2}}{m-1-s}$$
we have 
$$\ta{i}\left(R(i,j,r,m|t)\right)=\eta_{L_i}^{-r\frac{L_i}{l_i}}\sum_{s=0}^{m-1}\binom{\frac{1}{l_i}}{s}R\left(i,j,r,m-s|t-\frac{1}{l_i}\right)\in I_L^T.$$

Finally, if $i=j$, we have:
\begin{eqnarray*}
\lefteqn{\ta{i}\left(R(i,i,r,m|t)\right)}\\
&=&\sum_{\substack{n_1+n_2=-t\\n_1,n_2\in Z_i, \\n_1,n_2<0}}\eta_{L_i}^{rn_1L_i}\binom{-n_1-\frac{\left< \a^{(i)}, \a^{(i)} \right>}{2}}{m-1}x_{\a^{(i)}}^{\tnu}\left(n_1+\frac{1}{l_i}\right)x_{\a^{(i)}}^{\tnu}\left(n_2+\frac{1}{l_i}\right)\\
&=&\eta_{L_i}^{-r\frac{L_i}{l_i}}\sum_{\substack{n_1+n_2=-t+\frac{2}{l_i}\\n_1,n_2\in Z_i, \\n_1,n_2<0}}\eta_{L_i}^{rn_1L_i}\binom{-n_1-\frac{\left< \a^{(i)}, \a^{(i)} \right>}{2}+\frac{1}{l_i}}{m-1}x_{\a^{(i)}}^{\tnu}\left(n_1\right)x_{\a^{(i)}}^{\tnu}(n_2)\\
 && + cx_{\a^{(i)}}^{\tnu}\left(-\frac{\left< \alpha^{(i)}_{(0)},\alpha^{(i)}_{(0)} \right>}{2}+ \frac{1}{l_i}\right)x_{\a^{(i)}}^{\tnu}\left(-t+\frac{1}{l_i} + \frac{\left< \alpha^{(i)}_{(0)},\alpha^{(i)}_{(0)} \right>}{2}\right)\\
&=&\eta_{L_i}^{-r\frac{L_i}{l_i}}\sum_{k=0}^{m-1}\binom{\frac{1}{l_i}}{k}R\left(i,j,m-k|t-\frac{2}{l_i}\right) +  cx_{\a^{(i)}}^{\tnu}\left(-\frac{\left<\a_{(0)}^{(i)},\a_{(0)}^{(i)}\right>}{2} + \frac{1}{l_i}\right)x_{\a^{(i)}}^{\tnu}\left(-t+\frac{1}{l_i} +\frac{\left<\a_{(0)}^{(i)},\a_{(0)}^{(i)}\right>}{2}\right)
\end{eqnarray*}
for some constant $c \in \mathbb{C}$.
Observe that 
\begin{align*}x_{\a^{(i)}}^{\tnu}\left(-\frac{\left<\a_{(0)}^{(i)},\a_{(0)}^{(i)}\right>}{2} + \frac{1}{l_i}\right)&x_{\a^{(i)}}^{\tnu}\left(-t+\frac{1}{l_i} +\frac{\left<\a_{(0)}^{(i)},\a_{(0)}^{(i)}\right>}{2}\right)\\&= x_{\a^{(i)}}^{\tnu}\left(-t+\frac{1}{l_i} +\frac{\left<\a_{(0)}^{(i)},\a_{(0)}^{(i)}\right>}{2}\right)x_{\a^{(i)}}^{\tnu}\left(-\frac{\left<\a_{(0)}^{(i)},\a_{(0)}^{(i)}\right>}{2} + \frac{1}{l_i}\right)\in U_L^{T+} \subset I_L^T\end{align*} and that $\sum_{k=0}^{m-1}\binom{\frac{1}{l_i}}{k}R\left(i,j,m-k|t-\frac{2}{l_i}\right) \in I_L^T$.
It is easy to see that $\tau(U_L^{T+}) \subset I_L^T$.

Finally, we have that 
\begin{align*}
\ta{i}\left(x_{\a^{(i)}}^{\tnu}\left(-\frac{\left< \a^{(i)}_{(0)},\a^{(i)}_{(0)}\right>}{2}\right)\right) = x_{\a^{(i)}}^{\tnu}\left(-\frac{\left<\a^{(i)}_{(0)},\a^{(i)}_{(0)}\right>}{2}+ \frac{1}{l_i}\right)\in {U_L^T}x_{\a^{(i)}}^{\tnu}\left(-\frac{\left<\a^{(i)}_{(0)},\a^{(i)}_{(0)}\right>}{2}+ \frac{1}{l_i}\right)\subset I_L^T.
\end{align*}
 Since $I_L^T+U_L^Tx_{\a^{(i)}}^{\tnu}\left(-\frac{\left<\a^{(i)}_{(0)},\a^{(i)}_{(0)}\right>}{2}\right)$ is the left ideal of $U_L^T$ generated by the terms examined above, we have that 
$$
\ta{j}\left(I_L^T+U_L^Tx_{\a^{(i)}}^{\tnu}\left(-\frac{\left<\a^{(i)}_{(0)},\a^{(i)}_{(0)}\right>}{2}\right)\right)\subset I_L^T.$$

For the reverse inclusion, we consider the action of $\tau_{-\gamma_i}$, the inverse of $\tau_{\gamma_i}$, on $I_L^T$.
First, we note that $\tau_{-\gamma_i}(R(j,k,r,m|t)) = R(j,k,r,m|t)$ if $j\neq i$ and $k \neq i$.
Next, if $i\neq j$, we have
\begin{align*}
\tau_{-\gamma_i}\left(R(j,i,r,m|t)\right)&=\sum_{\substack{n_1+n_2=-t\\n_1\in Z^-_j, n_2\in Z^-_i \\n_1,n_2<0}}\eta_{L_j}^{rn_1L_j}\binom{-n_1-\frac{\left< \a^{(j)}, \a^{(j)} \right>}{2}}{m-1}x_{\a^{(j)}}^{\tnu}\left(n_1\right)x_{\a^{(i)}}^{\tnu}(n_2-\frac{1}{l_i})\\
&=\sum_{\substack{n_1+n_2=-t-\frac{1}{l_i}\\n_1\in Z^-_j, n_2\in Z^-_i\\n_1,n_2<0}}\eta_{L_j}^{rn_1L_j}\binom{-n_1-\frac{\left< \a^{(j)}, \a^{(j)} \right>}{2}}{m-1}x_{\a^{(j)}}^{\tnu}\left(n_1\right)x_{\a^{(i)}}^{\tnu}(n_2)\\
& \indent + cx_{\a^{(j)}}^{\tnu}\left(-t+\frac{\left< \alpha^{(j)}_{(0)},\alpha^{(j)}_{(0)} \right>}{2}-\frac{1}{l_i}\right)x_{\a^{(i)}}^{\tnu}\left(-\frac{\left< \alpha^{(i)}_{(0)},\alpha^{(i)}_{(0)} \right>}{2} \right)\\
&=R\left(j,i,r,m|t+\frac{1}{l_i}\right) + cx_{\a^{(j)}}^{\tnu}\left(-t+\frac{\left< \alpha^{(i)}_{(0)},\alpha^{(i)}_{(0)} \right>}{2}-\frac{1}{l_i}\right)x_{\a^{(i)}}^{\tnu}\left(-\frac{\left< \alpha^{(i)}_{(0)},\alpha^{(i)}_{(0)} \right>}{2} \right)
\end{align*}
for some constant $c\in \mathbb{C}$, and thus $ \tau_{-\gamma_i}\left(R(j,i,r,m|t)\right) \in I_L^T+U_L^Tx_{\a^{(i)}}^{\tnu}\left(-\frac{\left< \alpha^{(i)}_{(0)},\alpha^{(i)}_{(0)} \right>}{2}\right)$.
We also have
\begin{align*}
\tau_{-\gamma_i}\left(R(i,j,r,m|t)\right)&=\sum_{\substack{n_1+n_2=-t\\n_1\in Z^-_i, n_2\in Z^-_j \\n_1,n_2<0}}\eta_{L_i}^{rn_1L_i}\binom{-n_1-\frac{\left< \a^{(i)}, \a^{(i)} \right>}{2}}{m-1}x_{\a^{(i)}}^{\tnu}\left(n_1-\frac{1}{l_i}\right)x_{\a^{(j)}}^{\tnu}(n_2)\\
&=\eta_{L_i}^{\frac{rL_i}{l_i}}\sum_{\substack{n_1+n_2=-t-\frac{1}{l_i}\\n_1\in Z^-_i , n_2\in Z^-_j \\n_1< - \frac{1}{l_i},n_2<0}}\eta_{L_i}^{rn_1L_i}\binom{-n_1-\frac{\left< \a^{(i)}, \a^{(i)} \right>}{2}-\frac{1}{l_i}}{m-1}x_{\a^{(i)}}^{\tnu}\left(n_1\right)x_{\a^{(j)}}^{\tnu}(n_2),\end{align*}
Now using
$$\binom{-n_1-\frac{\left< \a^{(i)}, \a^{(i)} \right>}{2}-\frac{1}{l_i}}{m-1}=\sum_{k=0}^{m-1}\binom{-\frac{1}{l_i}}{k}\binom{-n_1-\frac{\left< \a^{(i)}, \a^{(i)} \right>}{2}}{m-1-k}$$
we have 
\begin{eqnarray*}
\lefteqn{\tau_{-\gamma_i}\left(R(i,j,r,m|t)\right)}\\
&=&\eta_{L_i}^{\frac{rL_i}{l_i}}\sum_{k=0}^{m-1}\binom{-\frac{1}{l_i}}{k}R\left(i,j,r,m-k|t+\frac{1}{l_i}\right) + c x_{\a^{(j)}}^{\tnu}\left(-t-\frac{1}{l_i}-\frac{\left< \alpha^{(i)}_{(0)},\alpha^{(i)}_{(0)} \right>}{2}\right)x_{\a^{(i)}}^{\tnu}\left(-\frac{\left< \alpha^{(i)}_{(0)},\alpha^{(i)}_{(0)} \right>}{2}\right)\\
&\in& I_L^T+U_L^Tx_{\a^{(i)}}^{\tnu}\left(-\frac{\left< \alpha^{(i)}_{(0)},\alpha^{(i)}_{(0)} \right>}{2}\right)
\end{eqnarray*}
for some constant $c\in \mathbb{C}$. Finally, we note that the case where $i=j$ is similar to the $i=j$ case above, and that $\tau_{\gamma_i}(U_L^{T+}) \subset I_L^T+U_L^Tx_{\a^{(i)}}^{\tnu}\left(-\frac{\left< \alpha^{(i)}_{(0)},\alpha^{(i)}_{(0)} \right>}{2}\right)$ is clear. Since $I_L^T$ is the left ideal generated by the above terms, we have the reverse inclusion
\be
\tau_{-\gamma_i}(I_L^T) \subset I_L^T+U_L^Tx_{\a^{(i)}}^{\tnu}\left(-\frac{\left< \alpha^{(i)}_{(0)},\alpha^{(i)}_{(0)} \right>}{2}\right)
\ee
or in other words
\be
I_L^T \subset \tau_{\gamma_i}\left(I_L^T+U_L^Tx_{\a^{(i)}}^{\tnu}\left(-\frac{\left< \alpha^{(i)}_{(0)},\alpha^{(i)}_{(0)} \right>}{2}\right)\right)
\ee
\end{proof}

We also define maps $\psi_{\gamma_j}:U_L^T\to U_L^T$ for $1\leq j\leq d$ by
\be\label{definepsi}
\psi_{\gamma_j}(a)=\tau_{-\a^{(j)}}(a)x_{\a^{(j)}}^{\tnu}\left(-\frac{\left<\a_{(0)}^{(j)},\a_{(0)}^{(j)}\right>}{2}\right)
\ee

\begin{lem}\label{lemma2} We have 
$$\tp{i}\left(I_L^T+U_L^T x_{\a^{(i)}}^{\tnu}\left(-\frac{\left<\a_{(0)}^{(i)},\a_{(0)}^{(i)}\right>}{2}\right)\right)\subset I_L^T.$$
\end{lem}

\begin{proof}

In light of Lemma \ref{lemma1}, it suffices to show that 
$$\psi_{\gamma_i}(I_L^T) \subset I_L^T$$.

We first observe that, by repeated applications of Lemma \ref{lemma1}, we have that
\begin{align}
\tau_{-\gamma_i}^s(I_L^T) \subset I_L^T + \sum_{r=0}^{s-1} U_L^Tx_{\a^{(i)}}^{\tnu}\left(-\frac{\left<\a_{(0)}^{(i)},\a_{(0)}^{(i)}\right>}{2}-\frac{r}{l_i}\right)
\end{align}
for any $s \ge 1$. We also note that
\begin{align}
\psi_{\gamma_i}\left( x_{\alpha^{(j)}}^{\tnu}\left({m}\right)\right) = \tau_{-\gamma_j}^{l_j\left< \alpha^{(i)}_{(0)},\alpha^{(j)}_{(0)} \right>} \left( x_{\alpha^{(j)}}^{\tnu}\left({m}\right)\right) x_{\alpha^{(i)}}^{\tnu}\left(-\frac{\left<\a_{(0)}^{(i)},\a_{(0)}^{(i)}\right>}{2}\right)
\end{align}
It suffices to analyze the effect of $\psi_{\gamma_i}$ on the generators of $I_L^T$, namely each $R(j,k,r,m|t)$ and $U_L^{T+}$. First, in the case that $j \neq k$, we note that:
\begin{eqnarray*}
\lefteqn{\psi_{\gamma_i}(R(j,k,r,m|t))}\\
&=& \tau_{-\gamma_j}^{l_j\left< \alpha^{(i)}_{(0)},\alpha^{(j)}_{(0)}\right>}\tau_{-\gamma_k}^{l_k\left< \alpha^{(i)}_{(0)},\alpha^{(k)}_{(0)}\right>}(R(j,k,r,m|t))x_{\alpha^{(i)}}^{\tnu}\left(-\frac{\left<\a_{(0)}^{(i)},\a_{(0)}^{(i)}\right>}{2}\right)\\
&\in& \tau_{-\gamma_j}^{l_j\left< \alpha^{(i)}_{(0)},\alpha^{(j)}_{(0)}\right>}(I_L^T + \sum_{r=0}^{{l_k\left< \alpha^{(i)}_{(0)},\alpha^{(k)}_{(0)}\right>}-1} U_L^Tx_{\a^{(k)}}^{\tnu}\left(-\frac{\left<\a_{(0)}^{(k)},\a_{(0)}^{(k)}\right>}{2}-\frac{r}{l_k}\right)x_{\alpha^{(i)}}^{\tnu}\left(-\frac{\left<\a_{(0)}^{(i)},\a_{(0)}^{(i)}\right>}{2}\right)\\
&\subset& \left( I_L^T + \sum_{r=0}^{{l_j\left< \alpha^{(i)}_{(0)},\alpha^{(j)}_{(0)}\right>}-1} U_L^Tx_{\a^{(j)}}^{\tnu}\left(-\frac{\left<\a_{(0)}^{(j)},\a_{(0)}^{(j)}\right>}{2}-\frac{r}{l_j}\right) \right)x_{\alpha^{(i)}}^{\tnu}\left(-\frac{\left<\a_{(0)}^{(i)},\a_{(0)}^{(i)}\right>}{2}\right)\\
&&+\left( \sum_{r=0}^{{l_k\left< \alpha^{(i)}_{(0)},\alpha^{(k)}_{(0)}\right>}-1} U_L^Tx_{\a^{(k)}}^{\tnu}\left(-\frac{\left<\a_{(0)}^{(k)},\a_{(0)}^{(k)}\right>}{2}-\frac{r}{l_k}\right) \right)x_{\alpha^{(i)}}^{\tnu}\left(-\frac{\left<\a_{(0)}^{(i)},\a_{(0)}^{(i)}\right>}{2}\right).
\end{eqnarray*}
Similarly, if $j=k$, we have
\begin{align}
\psi_{\gamma_i}(R(j,j,r,m|t)) \in \left( I_L^T + \sum_{r=0}^{{l_j\left< \alpha^{(i)}_{(0)},\alpha^{(j)}_{(0)}\right>}-1} U_L^Tx_{\a^{(j)}}^{\tnu}\left(-\frac{\left<\a_{(0)}^{(j)},\a_{(0)}^{(j)}\right>}{2}-\frac{r}{l_j}\right) \right)x_{\alpha^{(i)}}^{\tnu}\left(-\frac{\left<\a_{(0)}^{(i)},\a_{(0)}^{(i)}\right>}{2}\right).
\end{align}
In both cases, we apply Lemma \ref{NewRelations} and have that
\begin{align}
x_{\a^{(j)}}^{\tnu}\left(-\frac{\left<\a_{(0)}^{(j)},\a_{(0)}^{(j)}\right>}{2}-\frac{r}{l_j}\right)x_{\alpha^{(i)}}^{\tnu}\left(-\frac{\left<\a_{(0)}^{(i)},\a_{(0)}^{(i)}\right>}{2}\right) \in I_L^T
\end{align}
and
\begin{align}
x_{\a^{(j)}}^{\tnu}\left(-\frac{\left<\a_{(0)}^{(k)},\a_{(0)}^{(k)}\right>}{2}-\frac{r}{l_j}\right)x_{\alpha^{(i)}}^{\tnu}\left(-\frac{\left<\a_{(0)}^{(i)},\a_{(0)}^{(i)}\right>}{2}\right) \in I_L^T
\end{align}
for each $0 \le r \le l_j\left< (\alpha^{(i)})_{(0)},(\alpha^{(j)})_{(0)}\right> -1$.
Finally, we note that $\psi_{\gamma_i}(U_L^{T+}) \subset I_L^T$ also follows immediately from Lemma \ref{NewRelations}.

\end{proof}

Now consider the maps 
\be
e_{\a^{(i)}}: V_L^T\to V_L^T\ee
and their restrictions to $W_L^T\subset V_L^T$ where 
\be 
e_{\a^{(i)}}\cdot 1_T= k^{-\left< \a^{(i)},\a^{(i)}\right> /2} \sigma(\a^{(i)}) \left(e^{\a^{(i)}}\right)_{-\frac{\left<\a^{(i)}_{(0)},\a^{(i)}_{(0)}\right>}{2}}^{\hnu}\cdot 1_T 
\ee
where $\sigma$ is defined in (\ref{constant}), and 
\be
e_{\a^{(i)}}(e^{\beta})^{\hnu}_n=C(\a^{(i)},\beta)(e^{\beta})^{\hnu}_{n-\left<\beta_{(0)},\a^{(i)}\right>}e_{\a^{(i)}}\cdot 1_T\ee

Now lift this to a map $\widehat{e_{\a^{(i)}}}:U_L^T\to U_L^T$ so that 
\be
\begin{tikzcd}
U_L^T \arrow{r}{\widehat{e_{\a^{(i)}}}} \arrow[swap]{d}{f_L^T} & U_L^T \arrow{d}{f_L^T} \\
W_L^T  \arrow{r}{e_{\a^{(i)}}} & W_L^T
\end{tikzcd}
\ee
is a commutative diagram. Observe that
\be 
\widehat{e_{\a^{(i)}}}\cdot 1=k^{-\left< \a^{(i)},\a^{(i)}\right> /2} \sigma(\a^{(i)}) x_{\a^{(i)}}^{\tnu}\left(-\frac{\left<\a^{(i)}_{(0)},\a^{(i)}_{(0)}\right>}{2}\right)\cdot 1
\ee
and 
\be
\widehat{e_{\a^{(i)}}}x_{\a^{(j)}}^{\tnu}(n)=C(\a^{(i)},\a^{(j)})x_{\a^{(j)}}^{\tnu}\left(n-\left<(\a^{(j)})_{(0)},\a^{(i)}\right>\right)\widehat{e_{\a^{(i)}}}\ee

Recall that we have a basis dual to the $\ZZ$-basis with respect to $\left<\cdot,\cdot\right>$, of $L$ given by $\{\lambda_i |1\leq i\leq D\}$. For $\lambda\in\{\lambda^{(i)} |1\leq i\leq d\}$, we now introduce the map 
 \be
 \Delta^{T}(\lambda^{(i)},-x)=(\prod_{j=0}^{l_i-1} (-\eta_{l_i}^j)^{-\nu^j \lambda^{(i)}})x^{\lambda^{(i)}_{(0)}}E^{+}(-\lambda^{(i)},x).
 \ee
 This is a twisted version of the $\Delta$-map in \cite{Li}.  We will make use the constant term of $\Delta^T(\lambda^{(i)},-x)$ denoted by $\Delta_c^T(\lambda^{(i)},-x)$. We now explore the action of $\Delta_c^T(\lambda^{(i)},-x)$ on the operators $x_\alpha(m)$ where $\alpha\in\Delta_+$.\\
 For notational convenience, we write
\be
a \cdot 1_T = f_L^T(a)
\ee
for $a \in U_L^T$.
Using computations similar to those found in \cite{CalLM4}, \cite{CalMPe}, and \cite{PS1}-\cite{PS2}, we have:
 \begin{equation}\begin{aligned}
 \Delta_c^T(\lambda^{(i)},-x)(\ea{}{m}\cdot 1)&=\text{Coeff}_{x^0x_1^{-m-1}}\left(\Delta^T(\lambda^{(i)},-x)Y^{\hnu}(e^{\a},x_1)\right)\cdot1_T\\
 &=\text{Coeff}_{x_1^{\left<\lambda^{(i)},\a_{(0)}\right>-m-1}}Y^{\hnu}(e^{\a},x_1)\cdot 1_T\\
 &=x^{\tnu}_\alpha\left(m+\left<\lambda^{(i)},\a_{(0)}\right>\right)\cdot 1_T\\
 &=\ta{i}(x^{\tnu}_\alpha(m))\cdot 1_T,
\end{aligned}\end{equation}
  \color{black}
  It follows that we have linear maps
 \begin{equation}\label{DeltaProperty}\begin{aligned}
 \Delta_c^T(\lambda^{(i)},-x):W_L^T&\to W_L^T\\
 a\cdot 1_T&\mapsto \ta{i}(a)\cdot 1_T,
 \end{aligned}\end{equation}
 where $a\in U_L^T$.
We note here that the map $\ta{i}:U_L^T \rightarrow U_L^T$ is a lifting of the map $\Delta_c^T(\lambda^{(i)},-x)$ so that
\be
\begin{tikzcd}
U_L^T \arrow{rr}{\ta{i}} \arrow[swap]{d}{f_L^T} && U_L^T \arrow{d}{f_L^T} \\
W_L^T  \arrow{rr}{\Delta_c^T(\lambda^{(i)},-x)} && W_L^T
\end{tikzcd}
\ee
is a commutative diagram.

\subsection{A presentation for $W_L^T$}~\\

Now that we have developed the necessary tools, we are ready to  prove that $W_L^T$ has the expected presentation (cf. \cite{CalLM1}-\cite{CalLM4}, \cite{CalMPe}, \cite{PS1}-\cite{PS2} for the analogous presentations). 

\begin{thm}\label{presentation} We have
$$\text{Ker }f_L^T = I_L^T.$$

\end{thm}

\begin{proof} The inclusion
\be
I_L^T  \subset \mathrm{Ker} f_L^T \ 
\ee
is trivial. The remainder of the proof will be for the reverse inclusion.
Suppose that $a \in \ker f_L^T \setminus I_L^T$. We may assume that $a$ is homogeneous with respect to all gradings and is a nonzero such element with smallest 
possible total charge. (Consequently, $a$ has a nonzero charge for some $(\lambda^{(i)})_{(0)}$.) Among all elements of smallest possible total charge, we may assume that $a$ is also of smallest $L^{\hnu} (0)$-weight.
We first show that
\begin{equation}
a \in I_L^T + U_L^T x_{\a^{(i)}}^{\tnu}\left(-\frac{\left<\a^{(i)}_{(0)},\a^{(i)}_{(0)}\right>}{2}\right)
\end{equation}
Suppose not. Then
 by Lemma \ref{lemma1} we have that 
\begin{equation}
\ta{i}(a)  \notin I_L^T .
\end{equation}
We also have that 
\begin{equation}
a \cdot 1_T  = 0,
\end{equation}
and so using (\ref{DeltaProperty}), we have
\begin{equation}
\ta{i}(a) \cdot 1_T = 0.
\end{equation}
But, since 
\begin{equation}
wt(\ta{i}(a)) < wt(a)
\end{equation}
we have $\ta{i}(a) \in I_L^T$, a contradiction, and so
\begin{equation}
a \in I_L^T +U_L^T x_{\a^{(i)}}^{\tnu}\left(-\frac{\left<\a^{(i)}_{(0)},\a^{(i)}_{(0)}\right>}{2}\right).
\end{equation}

Hence, we may assume that $a \in I_L^T +U_L^T x_{\a^{(i)}}^{\tnu}\left(-\frac{\left<\a^{(i)}_{(0)},\a^{(i)}_{(0)}\right>}{2}\right)$, so that
$$
a = b + cx_{\a^{(i)}}^{\tnu}\left(-\frac{\left<\a^{(i)}_{(0)},\a^{(i)}_{(0)}\right>}{2}\right)
$$
for some $b \in I_L^T$ and $c \in U_L^T$. Notice that $b$ and $cx_{\a^{(i)}}^{\tnu}\left(-\frac{\left<\a^{(i)}_{(0)},\a^{(i)}_{(0)}\right>}{2}\right)$ are of the same total $L^{\hnu}(0)$ weight and total charge as $a$.  We have that 
\begin{eqnarray*}
c x_{\a^{(i)}}^{\tnu}\left(-\frac{\left<\a^{(i)}_{(0)},\a^{(i)}_{(0)}\right>}{2}\right)\cdot 1_T = (a-b) \cdot 1_T =0,
\end{eqnarray*}
so that $cx_{\a^{(i)}}^{\tnu}\left(-\frac{\left<\a^{(i)}_{(0)},\a^{(i)}_{(0)}\right>}{2}\right) \in \ker f_L^T \setminus I_L^T $. We also have that 
\begin{equation} 
cx_{\a^{(i)}}^{\tnu}\left(-\frac{\left<\a^{(i)}_{(0)},\a^{(i)}_{(0)}\right>}{2}\right)\cdot 1_T = \widehat{e_{\alpha^{(i)}}}(\tau_{\alpha^{(i)}}(c)) \cdot 1_T = e_{\alpha^{(i)}}(f_L^T(\tau_{\alpha^{(i)}}(c))) = 0
\end{equation}
so that, by the injectivity of ${e_{\alpha^{(i)}}}$, we have that 
\begin{equation}
f_L^T(\tau_{\alpha^{(i)}}(c))=\tau_{\alpha^{(i)}}(c) \cdot 1_T = 0.
\end{equation}
Since  $\tau_{\alpha^{(i)}}(c)$ has lower total charge than $a$, we have that 
\begin{equation}
\tau_{\alpha^{(i)}}(c) \in I_L^T.
\end{equation}
In view of Lemma \ref{lemma1}, we have that 
\begin{equation}
\tau_{-\gamma_i}(\tau_{\alpha^{(i)}}(c)) \in I_L^T + U_L^T x_{\a^{(i)}}^{\tnu}\left(-\frac{\left<\a^{(i)}_{(0)},\a^{(i)}_{(0)}\right>}{2}\right).
\end{equation}
Finally, applying Lemma \ref{lemma2} immediately gives
\begin{equation}
\psi_{\gamma_i}\tau_{\gamma_i}\tau_{-\gamma_i}(\tau_{\alpha^{(i)}}(c)) = cx_{\a^{(i)}}^{\tnu}\left(-\frac{\left<\a^{(i)}_{(0)},\a^{(i)}_{(0)}\right>}{2}\right) \in I_L^T.
\end{equation}
Thus, we have that 
\begin{equation}
a = b + cx_{\a^{(i)}}^{\tnu}\left(-\frac{\left<\a^{(i)}_{(0)},\a^{(i)}_{(0)}\right>}{2}\right) \in I_L^T,
\end{equation}
which is a contradiction, completing our proof.\\
\end{proof}

\subsection{Exact Sequences and Graded Dimensions}~

 \begin{thm}\label{exact}For each  $i=1,\dots,d$, we have the following short exact sequences
\begin{equation}
0\to W_L^T \xrightarrow{e_{\alpha^{(i)}}}W_L^T \xrightarrow{\Delta_c^T(\lambda^{(i)},-x)}W_L^T\to 0
\end{equation}

\end{thm}

\begin{proof}
It is clear that each $e_{\a^{(i)}}$ is injective and $\Delta_c^T(\lambda^{(i)},-x)$ is surjective. Let $w=a\cdot 1\in\text{ker }\Delta_c^T(\lambda^{(i)},-x)$ for some $a\in U_L^T$. So we have
\be
0=\Delta_c^T(\lambda^{(i)}, -x)w=\ta{i}(a)\cdot 1_T,
\ee 
and thus $\ta{i}(a)\in I_{\Lambda}$ by Theorem \ref{presentation}. This implies that 
\be
a\in I_L^T+U_L^T\xa{i}{-\frac{\left<\a^{(i)}_{(0)},\a^{(i)}_{(0)}\right>}{2}}
\ee
 by Lemma \ref{lemma1}. So we have the condition that $w=a\cdot 1\in\text{ker }\Delta_c^T(\lambda^{(i)},-x)$ is equivalent to the condition that $a\in I_L^T+U_L^T\xa{i}{-\frac{\left<\a^{(i)}_{(0)},\a^{(i)}_{(0)}\right>}{2}}$.

Now suppose $w=a\cdot 1\in\text{Im }e_{\a^{(i)}}$ for $a\in U_L^T$. So we can write $$w=b\xa{i}{-\frac{\left<\a^{(i)}_{(0)},\a^{(i)}_{(0)}\right>}{2}}\cdot 1_T$$ for $b\in U_L^T$ and thus $$a\cdot 1_T=b\xa{i}{-\frac{\left<\a^{(i)}_{(0)},\a^{(i)}_{(0)}\right>}{2}}\cdot 1_T\in I_L^T+U_L^T\xa{i}{-\frac{\left<\a^{(i)}_{(0)},\a^{(i)}_{(0)}\right>}{2}}.$$ This implies that the condition that $w=a\cdot 1_T\in\text{Im}e_{\a^{(i)}}$ is equivalent to the condition that $a\in I_L^T+U_L^T\xa{i}{-\frac{\left<\a^{(i)}_{(0)},\a^{(i)}_{(0)}\right>}{2}}$. So we have
\be
w=a\cdot 1_T\in\text{ker }\Delta_c^T(\lambda^{(i)},-x) \text{ if and only if } w=a\cdot 1_T\in\text{Im }e_{\a^{(i)}},
\ee
and thus 
\be
\text{ker }\Delta_c^T(\lambda^{(i)},-x)=\text{Im }e_{\a^{(i)}}.
\ee

\end{proof}
As a consequence, we have the following corollary:
\begin{cor}\label{exact2} We have the following short exact sequences for $i=1,\dots,d$:
\begin{eqnarray*}
0 &\to& \left(W_L^T\right)_{(\mathbf{m}-\epsilon_i,n+k\frac{\left<\a_{(0)}^{(i)},\a_{(0)}^{(i)}\right>}{2} -k\sum_{j=1}^d m_j \left<\a^{(i)}_{(0)},\a^{(j)}_{(0)}\right>)}\\
&& \xrightarrow{e_{\alpha^{(i)}}}\left(W_L^T\right)_{(\mathbf{m},n)}
 \xrightarrow{\Delta_c^T(\lambda^{(i)},-x)}\left(W_L^T\right)_{(\mathbf{m},n-\frac{k}{l_i}m_i)}\to 0.
\end{eqnarray*}
Moreover, we have recursions for $i=1,\dots,d$ of the form:
\begin{eqnarray}\label{recursion}
\lefteqn{\hspace{1in}\chi'(\mathbf{x};q)}\\
 \nonumber&=& \chi'(x_1,\dots,x_{i-1},q^{\frac{k}{l_i}}x_i,x_{i+1},\dots,x_d;q)\\
\nonumber&&+ x_i q^{k\frac{\left<\a_{(0)}^{(i)},\a_{(0)}^{(i)}\right>}{2}} \chi'(q^{k\left<\a^{(1)}_{(0)},\a^{(i)}_{(0)}\right>}x_1,\dots,q^{k\left<\a^{(d)}_{(0)},\a^{(i)}_{(0)}\right>}x_d;q).
\end{eqnarray}
\end{cor}

Setting 
\begin{equation}
\chi^{'}(\mathbf{x};q) = \sum_{m_1,\dots,m_d \in \mathbb{Z}_{\ge 0}} A_{m_1,\dots,m_d}(q)x_1^{m_1}\dots x_d^{m_d},
\end{equation}
we have from (\ref{recursion}):
\begin{equation}
A_{m_1,\dots, m_i+1,\dots,m_d}(q) = A_{m_1,\dots , m_d}(q)\frac{q^{k\frac{\left<\a_{(0)}^{(i)},\a_{(0)}^{(i)}\right>}{2} + \sum_{j=1}^d km_j \left< \alpha_{(0)}^{(j)}, \alpha_{(0)}^{(i)} \right>}}{1-q^{\frac{k(m_i+1)}{l_i}}}
\end{equation}
Solving these recursions with $A_{0,\dots,0}(q) = 1$ and using the notation $(a;q)_n = (1-a)(1-aq)(1-aq^2)\dots (1-aq^{n-1})$ (cf. \cite{A}), we have:

\begin{cor}\label{charater-eqn}
We have
\begin{equation}
\chi^{'}(\mathbf{x};q) = \sum_{{\bf m} \in (\mathbb{Z}_{\ge 0}^d)}\frac{q^{\frac{{\bf m}^t A {\bf m}}{2}}}{(q^{\frac{k}{l_1}};q^{\frac{k}{l_1}})_{m_1} \cdots (q^{\frac{k}{l_d}};q^{\frac{k}{l_d}})_{m_d} }x_1^{m_1}\cdots x_d^{m_d}
\end{equation}
where $A$ is the $(d\times d)$-matrix defined by 
$$A_{i,j} = k\left< \alpha_{(0)}^{(i)}, \alpha_{(0)}^{(j)}\right>.$$
\end{cor}

\begin{example}
Consider a rank $n+1$ lattice whose Gram matrix is the $(n+1) \times (n+1)$ matrix 
\be 
X_{n+1} = 
\begin{pmatrix}
2 & 1 & 1  & \cdots & 1\\
1 & 2 & 0  & \cdots & 0\\
1 & 0 & 2  & \cdots & 0\\
\vdots & \vdots & \vdots & \ddots &  \vdots\\
1 & 0 & 0  & \cdots & 2
\end{pmatrix}
\ee
Define the automorphism $\nu: L \rightarrow L$ by
\begin{equation*}
\nu(\alpha_1) = \alpha_1,
\end{equation*}
\begin{equation*}
\nu(\alpha_i) = \alpha_{i+1} \mbox{ for } 2 \le i \le n,
\end{equation*}
\begin{equation*}
\nu(\alpha_{n+1}) = \alpha_2.
\end{equation*}
Then we have that $k=2n$, $\alpha^{(1)}_{(0)} = \alpha_1$, and $\alpha^{(2)}_{(0)} = \frac{1}{n} \left(\alpha_2 + \dots + \alpha_{n+1}\right)$ and the $2 \times 2$ matrix $A$ in Corollary \ref{charater-eqn} is 
\begin{equation}
A = 2n(\left< \alpha_i^{(0)},\alpha_j^{(0)} \right> )_{i,j=1}^2 = \begin{pmatrix}
4n & 2n\\
2n & 4
\end{pmatrix}
\end{equation}
and is nonsingular when $n \neq 4$.
It follows that, replacing $q$ with $q^{\frac{1}{2}}$, we have:
\begin{equation}
\chi'(x_1, x_2,q) = \sum_{m_1,m_2 \ge 0} \frac{q^{nm_1^2 + nm_1m_2 + m_2^2}}{(q^n;q^n)_{m_1}(q;q)_{m_2}}x_1^{m_1}x_2^{m_2}.
\end{equation}
We note that in the case that $n=2$, we have 
\begin{equation}
\chi'(1, 1,q) = \sum_{m_1,m_2 \ge 0} \frac{q^{2m_1^2 + 2m_1m_2 + m_2^2}}{(q^2;q^2)_{m_1}(q;q)_{m_2}},
\end{equation}
which can be interpreted as the generating function of partitions of $n$ in which no part appears more than twice and no two parts differ by 1 (cf. \cite{Br}).
In the case that $n=3$, we have 
\begin{equation}
\chi'(1, 1,q) = \sum_{m_1,m_2 \ge 0} \frac{q^{3m_1^2 + 3m_1m_2 + m_2^2}}{(q^3;q^3)_{m_1}(q;q)_{m_2}}.
\end{equation}
We note that this is the analytic sum-side for the Kanade-Russell conjecture $I_1$ (cf. \cite{KR}), found in \cite{Ku}. Namely, one form of the Kanade-Russell Conjecture $I_1$ is:
\begin{equation}
\sum_{m_1,m_2 \ge 0} \frac{q^{3m_1^2 + 3m_1m_2 + m_2^2}}{(q^3;q^3)_{m_1}(q;q)_{m_2}}  = \frac{1}{(q,q^3,q^6,q^9;q)_\infty},
\end{equation}
where 
\be 
(a_1,\dots, a_j;q)_\infty = (a_1;q)_\infty \cdots (a_j,q)_\infty.
\ee
\end{example}

\section{Appendix -- A generalized Pascal Matrix}

While the majority of our results have followed directly from the theory of vertex operator algebras and their twisted modules, the proof of Lemma \ref{lemma2} relies on a purely linear algebra result which we will detail in this section. 

Fix nonnegative integers $N_0,N_1,\dots, N_{k-1}$ and $N$ such that $N=N_0+\cdots+N_{k-1}$ and $\eta$, a primitive $k^{\text{th}}$ root of unity. Also, let $z\in \mathbb{C}$ be arbitrary, $w\in\mathbb{C}^{\times}$ and for $1\leq r\leq k$ consider the matrices 
\be
A(r)=\begin{pmatrix} \binom{z}{0} & \eta^r\binom{z+w}{0} & \eta^{2r}\binom{z+2w}{0} & \cdots & \eta^{(N-1)r}\binom{z+(N-1)w}{0} \\[7pt]
\binom{z}{1} & \eta^r\binom{z+w}{1} & \eta^{2r}\binom{z+2w}{1} & \cdots & \eta^{(N-1)r}\binom{z+(N-1)w}{1} \\[7pt]
\binom{z}{2} & \eta^r\binom{z+w}{2} & \eta^{2r}\binom{z+2w}{2} & \cdots & \eta^{(N-1)r}\binom{z+(N-1)w}{2} \\[7pt]
\vdots & \vdots & \vdots & \ddots & \vdots \\[7pt]
\binom{z}{N_r-1} & \eta^r\binom{z+w}{N_r-1} & \eta^{2r}\binom{z+2w}{N_r-1} & \cdots & \eta^{(N-1)r}\binom{z+(N-1)w}{N_r-1} \end{pmatrix}.\ee

Now set
\be
A=\begin{pmatrix}A(0)\\A(1)\\\vdots \\ A(k-1)\end{pmatrix}.\ee

\begin{lem}
\label{lemmauppertriangularA}
For $x, z\in\mathbb{C}$, $w\in\mathbb{C}^{\times}$ and $p, q\in\mathbb{N}$, the matrix 
\be
A(x, z, w, p, q)=
\begin{pmatrix}
\binom{z}{0} & x\binom{z+w}{0} & x^{2}\binom{z+2w}{0} & \cdots & x^{q-1}\binom{z+(q-1)w}{0} \\[7pt]
\binom{z}{1} & x\binom{z+w}{1} & x^{2}\binom{z+2w}{1} & \cdots & x^{q-1}\binom{z+(q-1)w}{1} \\[7pt]
\binom{z}{2} & x\binom{z+w}{2} & x^{2}\binom{z+2w}{2} & \cdots & x^{q-1}\binom{z+(q-1)w}{2} \\[7pt]
\vdots & \vdots & \vdots & \ddots & \vdots \\[7pt]
\binom{z}{p-1} & x\binom{z+w}{p-1} & x^{2}\binom{z+2w}{p-1} & \cdots & x^{q-1}\binom{z+(q-1)w}{p-1} \end{pmatrix}\ee is row equivalent to the matrix 
\be
A'(x, p, q)=
\begin{pmatrix}
\binom{0}{0} & x\binom{1}{0} & x^{2}\binom{2}{0} & \cdots & x^{q-1}\binom{q-1}{0} \\[6pt]
\binom{0}{1} & x\binom{1}{1} & x^{2}\binom{2}{1} & \cdots & x^{q-1}\binom{q-1}{1} \\[6pt]
\binom{0}{2} & x\binom{1}{2} & x^{2}\binom{2}{2} & \cdots & x^{q-1}\binom{q-1}{2} \\[6pt]
\vdots & \vdots & \vdots & \ddots & \vdots \\[6pt]
\binom{0}{p-1} & x\binom{1}{p-1} & x^{2}\binom{2}{p-1} & \cdots & x^{q-1}\binom{q-1}{p-1} \end{pmatrix}.\ee
\end{lem}

\begin{proof}
For $1\leq n\leq p-2$, let 
$$
M(n)=
\begin{pmatrix}
\binom{0}{0} & \binom{1}{0} & \binom{2}{0} & \cdots & \binom{q-1}{0} \\[6pt]
\binom{0}{1} & \binom{1}{1} & \binom{2}{1} & \cdots & \binom{q-1}{1} \\[6pt]
\vdots & \vdots & \vdots & \ddots & \vdots \\[6pt]
\binom{0}{n} & \binom{1}{n} & \binom{2}{n} & \cdots & \binom{q-1}{n} \\[6pt]
\binom{-1}{n}\binom{0}{1} & \binom{0}{n}\binom{w}{1} & \binom{1}{n}\binom{2w}{1} & \cdots & \binom{q-2}{n}\binom{(q-1)w}{1} \\[6pt]
\binom{-1}{n}\binom{0}{2} & \binom{0}{n}\binom{w}{2} & \binom{1}{n}\binom{2w}{2} & \cdots & \binom{q-2}{n}\binom{(q-1)w}{2} \\[6pt]
\vdots & \vdots & \vdots & \ddots & \vdots \\[6pt]
\binom{-1}{n}\binom{0}{p-n-1} & \binom{0}{n}\binom{w}{p-n-1} & \binom{1}{n}\binom{2w}{p-n-1} & \cdots & \binom{q-2}{n}\binom{(q-1)w}{p-n-1} \end{pmatrix}.
$$
Furthermore, let 
$$
M=
\begin{pmatrix}
\binom{0}{0} & \binom{w}{0} & \binom{2w}{0} & \cdots & \binom{(q-1)w}{0} \\[6pt]
\binom{0}{1} & \binom{w}{1} & \binom{2w}{1} & \cdots & \binom{(q-1)w}{1} \\[6pt]
\binom{0}{2} & \binom{w}{2} & \binom{2w}{2} & \cdots & \binom{(q-1)w}{2} \\[6pt]
\vdots & \vdots & \vdots & \ddots & \vdots \\[6pt]
\binom{0}{p-1} & \binom{w}{p-1} & \binom{2w}{p-1} & \cdots & \binom{(q-1)w}{p-1} \end{pmatrix}.
$$
We will show, for $1\leq n\leq p-2$, there exists an invertible matrix $P(n)$ such that $P(n)M = M(n)$.  We proceed by induction.  For $0\leq n\leq p-3$, define 
$$
Q(n)=
\begingroup
\setlength{\arraycolsep}{3pt}
\begin{pmatrix}
1 & \cdots & 0 & 0 & 0 & 0 & \cdots & 0 & 0 \\[3pt]
\vdots & \ddots & \vdots & \vdots & \vdots & \vdots & \ddots & \vdots & \vdots \\[3pt]
0 & \cdots & 1 & 0 & 0 & 0 & \cdots & 0 & 0 \\[7pt]
0 & \cdots & 0 & \frac{1}{(n+1)w} & 0 & 0 & \cdots & 0 & 0 \\[7pt]
0 & \cdots & 0 & \frac{1-(n+1)w}{(n+1)w} & \frac{2}{(n+1)w} & 0 & \cdots & 0 & 0 \\[7pt]
0 & \cdots & 0 & 0 & \frac{2-(n+1)w}{(n+1)w} & \frac{3}{(n+1)w} & \cdots & 0 & 0 \\[7pt]
\vdots & \ddots & \vdots & \vdots & \vdots & \vdots & \ddots & \vdots & \vdots \\[7pt]
0 & \cdots & 0 & 0 & 0 & 0 & \cdots & \frac{p-2-n-(n+1)w}{(n+1)w} & \frac{p-1-n}{(n+1)w} \end{pmatrix}\endgroup,
$$ where the identity submatrix has size $n+1$.  Letting $P(1)=Q(0)$, we see that the $ij$th entry of $P(1)M$ is $\binom{jw}{0}=\binom{j}{0}$ for $i=0$, $\frac{1}{w}\binom{jw}{1}=j=\binom{j}{1}$ for $i=1$ and 
\begin{align*}
\frac{i-1-w}{w}\binom{jw}{i-1}+\frac{i}{w}\binom{jw}{i}&=\frac{i-1-w}{w}\binom{jw}{i-1}+\frac{i}{w}\left(\frac{jw-i+1}{i}\right)\binom{jw}{i-1} \\
&=\frac{i-1-w}{w}\binom{jw}{i-1}+\frac{jw-i+1}{w}\binom{jw}{i-1} \\
&=\frac{jw-w}{w}\binom{jw}{i-1}=\binom{j-1}{1}\binom{jw}{i-1}
\end{align*} for $1<i\leq p-1$.  Thus, $P(1)M = M(1)$.  Also, since $Q(n)$ is clearly invertible for all $0\leq n\leq p-3$, $P(1)$ is invertible.  

Now, assuming that $P(n)$ exists and satisfies the given criteria, let $P(n+1) = Q(n)P(n)$.  The $ij$th entry of $P(n+1)M=Q(n)P(n)M=Q(n)M(n)$ is $\binom{j}{i}$ for $0\leq i\leq n$, $\frac{1}{(n+1)w}\binom{j-1}{n}\binom{jw}{1}=\frac{j}{(n+1)}\binom{j-1}{n}=\binom{j}{n+1}$ for $i=n+1$ and 
\begin{align*}
\mathmakebox[9em][l]{\frac{i-n-1-(n+1)w}{(n+1)w}\binom{j-1}{n}\binom{jw}{i-1-n}+\frac{i-n}{(n+1)w}\binom{j-1}{n}\binom{jw}{i-n}} \\
&=\frac{i-n-1-(n+1)w}{(n+1)w}\binom{j-1}{n}\binom{jw}{i-1-n} \\
&\phantom{{}={}}+\frac{i-n}{(n+1)w}\left(\frac{jw-i+n+1}{i-n}\right)\binom{j-1}{n}\binom{jw}{i-n-1} \\
&=\frac{i-n-1-(n+1)w}{(n+1)w}\binom{j-1}{n}\binom{jw}{i-1-n} \\
&\phantom{{}={}}+\frac{jw-i+n+1}{(n+1)w}\binom{j-1}{n}\binom{jw}{i-1-n} \\
&=\frac{jw-(n+1)w}{(n+1)w}\binom{j-1}{n}\binom{jw}{i-1-n} \\
&=\frac{j-(n+1)}{n+1}\binom{j-1}{n}\binom{jw}{i-1-n}=\binom{j-1}{n+1}\binom{jw}{i-(n+1)}
\end{align*} for $n+1<i\leq p-1$.  Therefore, $P(n+1)M=M(n+1)$.  By assumption, $P(n)$ is invertible, and since $Q(n)$ is invertible, $P(n+1)$ is also invertible.  

Let 
$$
Q=
\begin{pmatrix}
1 & \cdots & 0 & 0 \\[3pt]
\vdots & \ddots & \vdots & \vdots \\[3pt]
0 & \cdots & 1 & 0 \\[7pt]
0 & \cdots & 0 & \frac{1}{(p-1)w} \end{pmatrix},
$$ where the identity submatrix has size $p-1$, and let $P=QP(p-2)$.  Then the $ij$th entry of $PM=QP(p-2)M=QM(p-2)$ is $\binom{j}{i}$ for $0\leq i\leq p-2$ and $\frac{1}{(p-1)w}\binom{j-1}{p-2}\binom{jw}{1}=\frac{j}{p-1}\binom{j-1}{p-2}=\binom{j}{p-1}$ for $i=p-1$, so 
\be
\label{matrixPM}
PM=
\begin{pmatrix}
\binom{0}{0} & \binom{1}{0} & \cdots & \binom{q-1}{0} \\[6pt]
\binom{0}{1} & \binom{1}{1} & \cdots & \binom{q-1}{1} \\[6pt]
\vdots & \vdots & \ddots & \vdots \\[6pt]
\binom{0}{p-1} & \binom{1}{p-1} & \cdots & \binom{q-1}{p-1} \end{pmatrix}.\ee  
By the Chu-Vandermonde identity, the $ij$th entry of $A(x, z, w, p, q)$ can be written as $$x^{j}\binom{z+jw}{i}=x^{j}\sum_{m=0}^{i}\binom{z}{m}\binom{jw}{i-m},$$ so letting 
$$
Z=
\begin{pmatrix}
\binom{z}{0} & 0 & 0 & \cdots & 0 \\[6pt]
\binom{z}{1} & \binom{z}{0} & 0 & \cdots & 0 \\[6pt]
\binom{z}{2} & \binom{z}{1} & \binom{z}{0} & \cdots & 0 \\[6pt]
\vdots & \vdots & \vdots & \ddots & \vdots \\[6pt]
\binom{z}{p-1} & \binom{z}{p-2} & \binom{z}{p-3} & \cdots & \binom{z}{0} \end{pmatrix}\text{, }
H=
\begin{pmatrix}
1 & 0 & \cdots & 0 \\[4pt]
0 & x & \cdots & 0 \\[4pt]
\vdots & \vdots & \ddots & \vdots \\[4pt]
0 & 0 & \cdots & x^{q-1} \end{pmatrix},
$$ we have 
\be
A(x, z, w, p, q)=ZMH.\ee  
Since both $Q$ and $P(p-2)$ are invertible, $P$ is invertible.  Clearly, $Z$ is also invertible.  Thus, $A(x, z, w, p, q)=ZP^{-1}PMH$, where $ZP^{-1}$ is invertible, implying that $A(x, z, w, p, q)$ is row equivalent to $PMH=A'(x, p, q)$.  
\end{proof}

\begin{lem}
\label{lemmatwoAprimes}
Suppose $x, y\in\mathbb{C}^{\times}$ and $n, s, t\in\mathbb{N}$ such that $x\neq y$ and $s\geq t$.  Furthermore, let 
$$
C=
\begin{pmatrix}
0 & \cdots & 0 & \binom{s}{s} & \binom{s+1}{s}x & \binom{s+2}{s}x^{2} & \cdots & \binom{n-1}{s}x^{n-s-1} \\[7pt]
0 & \cdots & 0 & 0 & \binom{s+1}{s+1} & \binom{s+2}{s+1}x & \cdots & \binom{n-1}{s+1}x^{n-s-2} \\[7pt]
0 & \cdots & 0 & 0 & 0 & \binom{s+2}{s+2} & \cdots & \binom{n-1}{s+2}x^{n-s-3} \\[7pt]
\vdots & \ddots & \vdots & \vdots & \vdots & \vdots & \ddots & \vdots \\[7pt]
0 & \cdots & 0 & 0 & 0 & 0 & \cdots & \binom{n-1}{n-1} \end{pmatrix},
$$ where the number of zero columns in $C$ is $s$.  Then, the matrix 
$$
B=
\begin{pmatrix}
A'(x, s, n) \\[2pt]
A'(y, t, n) \end{pmatrix}
$$ is row equivalent to the matrix 
$$
B'=
\begin{pmatrix}
A'(x, s, n) \\[2pt]
A'(y-x, t, n-s)C \end{pmatrix}.
$$  
\end{lem}

\begin{proof}
Let $Q$ be the $t\times s$ matrix whose $ij$th entry is $-\left(\frac{y}{x}\right)^{i}\binom{j}{i}\left(\frac{y-x}{x}\right)^{j-i}$, where $0\leq i\leq t-1$ and $0\leq j\leq s-1$.  Then, let 
$$
Q'=
\begin{pmatrix}
I_s & 0 \\
Q & I_t \end{pmatrix},
$$ where $I_s$ and $I_t$ denote the identity matrices of size $s$ and $t$, respectively.  For $0\leq i\leq s-1$, the $ij$th entry of $Q'B$ is the $ij$th entry of $A'(x, s, n)$.  For $s\leq i\leq s+t-1$, the $ij$th entry of $Q'B$ is 
\begin{align*}
q_{ij}&=\sum_{m=0}^{s-1}\left(-\left(\frac{y}{x}\right)^{i-s}\binom{m}{i-s}\left(\frac{y-x}{x}\right)^{m-i+s}\right)\left(x^j\binom{j}{m}\right)+y^j\binom{j}{i-s} \\
&=-y^{i-s}\sum_{m=0}^{s-1}\binom{m}{i-s}\binom{j}{m}(y-x)^{m-i+s}x^{j-m}+y^j\binom{j}{i-s} \\
&=-y^{i-s}\sum_{m=i-s}^{\min\{j, s-1\}}\binom{m}{i-s}\binom{j}{m}(y-x)^{m-i+s}x^{j-m}+y^j\binom{j}{i-s} \\
&=-y^{i-s}\sum_{m=i-s}^{j}\binom{m}{i-s}\binom{j}{m}(y-x)^{m-i+s}x^{j-m} \\
&\phantom{{}={}}+y^{i-s}\sum_{m=s}^{j}\binom{m}{i-s}\binom{j}{m}(y-x)^{m-i+s}x^{j-m}+y^j\binom{j}{i-s} \\
&=-y^{i-s}\binom{j}{i-s}\sum_{m=i-s}^{j}\binom{j-i+s}{m-i+s}(y-x)^{m-i+s}x^{j-m} \\
&\phantom{{}={}}+y^{i-s}\sum_{m=s}^{j}\binom{m}{i-s}\binom{j}{m}(y-x)^{m-i+s}x^{j-m}+y^j\binom{j}{i-s},
\end{align*} where we have used the identity $\binom{a}{b}=0$ for $0\leq a < b$, the fact that $i-s\leq t-1\leq s-1$ and the identity $\binom{a}{b}\binom{c}{a}=\binom{c}{b}\binom{c-b}{a-b}$ for $0\leq b\leq a\leq c$.  If $j < i-s$, $\binom{j}{i-s}=0$ and we have 
$$
q_{ij}=y^{i-s}\sum_{m=s}^{j}\binom{m}{i-s}\binom{j}{m}(y-x)^{m-i+s}x^{j-m}.
$$  Otherwise, $0\leq j-i+s$ and, by the binomial theorem, we have 
\begin{align*}
q_{ij}&=-y^{i-s}\binom{j}{i-s}((y-x)+x)^{j-i+s} \\
&\phantom{{}={}}+y^{i-s}\sum_{m=s}^{j}\binom{m}{i-s}\binom{j}{m}(y-x)^{m-i+s}x^{j-m}+y^j\binom{j}{i-s} \\
&=-y^{i-s}\binom{j}{i-s}y^{j-i+s} \\
&\phantom{{}={}}+y^{i-s}\sum_{m=s}^{j}\binom{m}{i-s}\binom{j}{m}(y-x)^{m-i+s}x^{j-m}+y^j\binom{j}{i-s} \\
&=-y^{j}\binom{j}{i-s}+y^{i-s}\sum_{m=s}^{j}\binom{m}{i-s}\binom{j}{m}(y-x)^{m-i+s}x^{j-m}+y^j\binom{j}{i-s} \\
&=y^{i-s}\sum_{m=s}^{j}\binom{m}{i-s}\binom{j}{m}(y-x)^{m-i+s}x^{j-m}.
\end{align*}  In either case, 
$$
q_{ij}=y^{i-s}\sum_{m=s}^{j}\binom{m}{i-s}\binom{j}{m}(y-x)^{m-i+s}x^{j-m}.
$$  Note that, if $j < s$, this sum is empty, so $q_{ij}=0$.  Thus, 
$$
Q'B=
\begin{pmatrix}
A'(x, s, n) \\[2pt]
WC \end{pmatrix},
$$ where 
$$
W=
\begingroup
\setlength{\arraycolsep}{2pt}
\begin{pmatrix}
\binom{s}{0}(y-x)^{s} & \binom{s+1}{0}(y-x)^{s+1} & \cdots & \binom{n-1}{0}(y-x)^{n-1} \\[7pt]
\binom{s}{1}y(y-x)^{s-1} & \binom{s+1}{1}y(y-x)^{s} & \cdots & \binom{n-1}{1}y(y-x)^{n-2} \\[7pt]
\vdots & \vdots & \ddots & \vdots \\[7pt]
\binom{s}{t-1}y^{t-1}(y-x)^{s-t+1} & \binom{s+1}{t-1}y^{t-1}(y-x)^{s-t+2} & \cdots & \binom{n-1}{t-1}y^{t-1}(y-x)^{n-t} \end{pmatrix}\endgroup.
$$  Furthermore, $W=VA(y-x, s, 1, t, n-s)$, where 
$$
V=
\begin{pmatrix}
(y-x)^{s} & 0 & 0 & \cdots & 0 \\[6pt]
0 & y(y-x)^{s-1} & 0 & \cdots & 0 \\[6pt]
0 & 0 & y^2(y-x)^{s-2} & \cdots & 0\\[6pt]
\vdots & \vdots & \vdots & \ddots & \vdots \\[6pt]
0 & 0 & 0 & \cdots & y^{t-1}(y-x)^{s-t+1} \end{pmatrix}.
$$  Letting 
$$
V'=
\begin{pmatrix}
I_{s} & 0 \\
0 & V \end{pmatrix},
$$ we have 
$$
Q'B=
\begin{pmatrix}
A'(x, s, n) \\[2pt]
VA(y-x, s, 1, t, n-s)C \end{pmatrix}=
V'\begin{pmatrix}
A'(x, s, n) \\[2pt]
A(y-x, s, 1, t, n-s)C \end{pmatrix}.
$$  By Lemma \ref{lemmauppertriangularA}, there exists an invertible matrix $U$ such that $UA(y-x, s, 1, t, n-s)=A'(y-x, t, n-s)$.  Therefore, letting 
$$
U'=
\begin{pmatrix}
I_{s} & 0 \\
0 & U^{-1} \end{pmatrix},
$$ we have 
$$
Q'B=
V'\begin{pmatrix}
A'(x, s, n) \\[2pt]
U^{-1}A'(y-x, t, n-s)C \end{pmatrix}=
V'U'\begin{pmatrix}
A'(x, s, n) \\[2pt]
A'(y-x, t, n-s)C \end{pmatrix}=V'U'B'.
$$  Since $x\neq y$, $y-x\neq 0$ and, thus, $V$ is invertible.  It follows that $V'$ is invertible.  Clearly, $Q'$ and $U'$ are also invertible.  Therefore, $(U')^{-1}(V')^{-1}Q'B=B'$, where $(U')^{-1}(V')^{-1}Q'$ is invertible, implying that $B$ is row equivalent to $B'$.  
\end{proof}

\begin{thm} \label{GeneralizedPascal}
The matrix $A$ is invertible.\end{thm}

\begin{proof}
Let $m, n\in\mathbb{N}$.  Then, let $S=\left(s_0, \ldots, s_{m-1}\right)\in\mathbb{N}^m$ such that $s_0+\cdots +s_{m-1}=n$.  Also, let $X=\left(x_0, \ldots, x_{m-1}\right)\in\left(\mathbb{C}^{\times}\right)^m$ such that $x_0, \ldots, x_{m-1}$ are distinct complex numbers.  Finally, let 
$$
Q(m, n, S, X)=
\begin{pmatrix}
A'\left(x_0, s_0, n\right) \\
A'\left(x_1, s_1, n\right) \\
\vdots \\
A'\left(x_{m-1}, s_{m-1}, n\right) \end{pmatrix}.
$$  Proceeding by induction on $m$, we will show that $Q(m, n, S, X)$ is invertible for all $m, n, S, X$.  

If $m=1$, $Q(m, n, S, X)=A'\left(x_0, n, n\right)$.  This matrix is an upper triangular matrix whose diagonal entries are $x_0^i\neq 0$.  Thus, $Q(m, n, S, X)$ is invertible.  

Now, let $m\in\mathbb{N}$ and assume that $Q\left(m, n', S', X'\right)$ is invertible for all $n'\in\mathbb{N}$, $S'\in\mathbb{N}^m$, $X'\in\left(\mathbb{C}^{\times}\right)^m$ satisfying the given conditions.  By reordering the rows of 
$$
Q(m+1, n, S, X)=
\begin{pmatrix}
A'\left(x_0, s_0, n\right) \\
A'\left(x_1, s_1, n\right) \\
\vdots \\
A'\left(x_{m}, s_{m}, n\right) \end{pmatrix},
$$ we observe that $Q(m+1, n, S, X)$ is row equivalent to the matrix 
$$
Q'(m+1, n, S, X)=
\begin{pmatrix}
A'\left(x_{\sigma(0)}, s_{\sigma(0)}, n\right) \\[3pt]
A'\left(x_{\sigma(1)}, s_{\sigma(1)}, n\right) \\[3pt]
\vdots \\[3pt]
A'\left(x_{\sigma(m)}, s_{\sigma(m)}, n\right) \end{pmatrix},
$$ where $\sigma$ is a permutation of the set $\{0, \ldots, m\}$ such that $s_{\sigma(0)}=\max\left\{s_0, \ldots, s_m\right\}$.  Then, by Lemma \ref{lemmatwoAprimes}, $Q'(m+1, n, S, X)$ is row equivalent to 
$$
Q''(m+1, n, S, X)=
\begin{pmatrix}
A'\left(x_{\sigma(0)}, s_{\sigma(0)}, n\right) \\[3pt]
A'\left(x_{\sigma(1)}-x_{\sigma(0)}, s_{\sigma(1)}, n-s_{\sigma(0)}\right)C \\[3pt]
A'\left(x_{\sigma(2)}-x_{\sigma(0)}, s_{\sigma(2)}, n-s_{\sigma(0)}\right)C \\[3pt]
\vdots \\[3pt]
A'\left(x_{\sigma(m)}-x_{\sigma(0)}, s_{\sigma(m)}, n-s_{\sigma(0)}\right)C \end{pmatrix},
$$ where $C=\begin{pmatrix}0 & C' \end{pmatrix}$ and 
$$
C'=
\begin{pmatrix}
\binom{s_{\sigma(0)}}{s_{\sigma(0)}} & \binom{s_{\sigma(0)}+1}{s_{\sigma(0)}}x_{\sigma(0)} & \binom{s_{\sigma(0)}+2}{s_{\sigma(0)}}x_{\sigma(0)}^{2} & \cdots & \binom{n-1}{s_{\sigma(0)}}x_{\sigma(0)}^{n-s_{\sigma(0)}-1} \\[7pt]
0 & \binom{s_{\sigma(0)}+1}{s_{\sigma(0)}+1} & \binom{s_{\sigma(0)}+2}{s_{\sigma(0)}+1}x_{\sigma(0)} & \cdots & \binom{n-1}{s_{\sigma(0)}+1}x_{\sigma(0)}^{n-s_{\sigma(0)}-2} \\[7pt]
0 & 0 & \binom{s_{\sigma(0)}+2}{s_{\sigma(0)}+2} & \cdots & \binom{n-1}{s_{\sigma(0)}+2}x_{\sigma(0)}^{n-s_{\sigma(0)}-3} \\[7pt]
\vdots & \vdots & \vdots & \ddots & \vdots \\[7pt]
0 & 0 & 0 & \cdots & \binom{n-1}{n-1} \end{pmatrix}.
$$  Let $n'=n-s_{\sigma(0)}$, $S'=\left(s_{\sigma(1)}, \ldots, s_{\sigma(m)}\right)$ and $$X'=\left(x_{\sigma(1)}-x_{\sigma(0)}, \ldots, x_{\sigma(m)}-x_{\sigma(0)}\right).$$  Note that, since $x_0, \ldots, x_{m}$ are distinct nonzero complex numbers, $$x_{\sigma(1)}-x_{\sigma(0)}, \ldots, x_{\sigma(m)}-x_{\sigma(0)}$$ are also distinct nonzero complex numbers.  Therefore, we may write 
$$
Q\left(m, n', S', X'\right)=
\begin{pmatrix}
A'\left(x_{\sigma(1)}-x_{\sigma(0)}, s_{\sigma(1)}, n'\right) \\[3pt]
A'\left(x_{\sigma(2)}-x_{\sigma(0)}, s_{\sigma(2)}, n'\right) \\[3pt]
\vdots \\[3pt]
A'\left(x_{\sigma(m)}-x_{\sigma(0)}, s_{\sigma(m)}, n'\right) \end{pmatrix}.
$$  Let $T'=Q\left(m, n', S', X'\right)C'$.  Then, 
\begin{align*}
Q''(m+1, n, S, X)&=
\begin{pmatrix}
A'\left(x_{\sigma(0)}, s_{\sigma(0)}, n\right) \\[3pt]
Q\left(m, n', S', X'\right)C \end{pmatrix} \\&=
\begin{pmatrix}
A'\left(x_{\sigma(0)}, s_{\sigma(0)}, n\right) \\[3pt]
Q\left(m, n', S', X'\right)\begin{pmatrix}0 & C' \end{pmatrix} \end{pmatrix}=
\begin{pmatrix}
A'\left(x_{\sigma(0)}, s_{\sigma(0)}, n\right) \\[3pt]
\begin{pmatrix}0 & T' \end{pmatrix} \end{pmatrix}.
\end{align*}  By assumption, $Q\left(m, n', S', X'\right)$ is invertible.  Since $C'$ is clearly invertible, $T'$ is invertible.  It follows that $Q''(m+1, n, S, X)$ is row equivalent to 
$$
Q'''(m+1, n, S, X)=
\begin{pmatrix}
A'\left(x_{\sigma(0)}, s_{\sigma(0)}, n\right) \\[3pt]
\begin{pmatrix}0 & I_{n'} \end{pmatrix} \end{pmatrix}.
$$  This matrix is an upper triangular matrix whose diagonal entries are $x_{\sigma(0)}^i\neq 0$ for $0\leq i < s_{\sigma(0)}$ and $1$ for $s_{\sigma(0)}\leq i < n$.  Thus, $Q'''(m+1, n, S, X)$ is invertible.  It follows that $Q(m+1, n, S, X)$ is invertible.  

Finally, by Lemma \ref{lemmauppertriangularA}, the matrix 
$$
A=
\begin{pmatrix}
A(0) \\
A(1) \\
\vdots \\
A(k-1) \end{pmatrix}=
\begin{pmatrix}
A(1, z, w, N_0, N) \\[2pt]
A(\eta, z, w, N_1, N) \\[2pt]
\vdots \\[2pt]
A(\eta^{k-1}, z, w, N_{k-1}, N) \end{pmatrix}
$$ is row equivalent to the matrix 
$$
A'=
\begin{pmatrix}
A'(1, N_0, N) \\[2pt]
A'(\eta, N_1, N) \\[2pt]
\vdots \\[2pt]
A'(\eta^{k-1}, N_{k-1}, N) \end{pmatrix}=Q\left(k, N, \left(N_0, N_1, \ldots, N_{k-1}\right), \left(1, \eta, \ldots, \eta^{k-1}\right)\right).
$$  Since $A'=Q\left(k, N, \left(N_0, N_1, \ldots, N_{k-1}\right), \left(1, \eta, \ldots, \eta^{k-1}\right)\right)$ is invertible, $A$ is invertible.  
\end{proof}

\vspace{.3in}

\vspace{.2in}

\noindent{\small \sc Department of Mathematics, Randolph College, 
Lynchburg, VA 24503} \\ {\em E--mail address}: mpenn@randolphcollege.edu

\vspace{.2in}
\noindent{\small \sc Department of Mathematics and Computer Science, Ursinus College, 
Collegeville, PA 19426} \\ {\em E--mail address}: csadowski@ursinus.edu

\vspace{.2in}
\noindent{\small \sc Department of Mathematics, University of Oregon, 
Eugene, OR 97403} \\ {\em E--mail address}: gwebb@uoregon.edu

\end{document}